\documentclass{article}
\usepackage[english]{babel}
\usepackage{amsmath,amssymb,latexsym}

\newcommand{\assign}{:=}
\newcommand{\tmem}[1]{{\em #1\/}}
\newcommand{\tmmathbf}[1]{\ensuremath{\boldsymbol{#1}}}
\newcommand{\tmop}[1]{\ensuremath{\operatorname{#1}}}
\newcommand{\tmstrong}[1]{\textbf{#1}}

\newenvironment{proof}{\noindent\textbf{Proof\ }}{\hspace*{\fill}$\Box$\medskip}
\newtheorem{proposition}{Proposition}
\newtheorem{remark}{Remark}
\newtheorem{claim}{Claim}

\usepackage{xcolor} 
\usepackage{graphicx,subfigure}
\usepackage{parskip}

\usepackage{geometry}
\title{The Onsager principle and {structure} preserving numerical schemes\thanks{The work of Huangxin Chen was partially supported by National Key Research and Development Project of China (Grant No. 2023YFA1011702) and National Natural Science Foundation of China (Grant No. {12471345 and No. 12122115}). The work of Xianmin Xu was partially supported by National Natural Science Foundation of China (Grant No. 11971469 and No. 12371415) and by Beijing Natural Science Foundation(Grant No. Z240001).}}

\author{Huangxin Chen\thanks{School of Mathematical Sciences and Fujian Provincial Key Laboratory on Mathematical Modeling and High Performance Scientific Computing, Xiamen University, Fujian, 361005, China (chx@xmu.edu.cn).}
\and
Hailiang Liu\thanks{Iowa State University, Department of Mathematics, Ames, IA 50011, United States (hliu@iastate.edu)}
\and
Xianmin Xu\thanks{Corresponding author. LSEC, ICMSEC, NCMIS, Academy of Mathematics and Systems Science, Chinese Academy of Sciences, Beijing 100190, China (xmxu@lsec.cc.ac.cn)}}

\date{}

\begin{document}

\maketitle

\begin{abstract}
We present a natural framework for constructing energy-stable time discretization schemes. By leveraging the Onsager principle,
we demonstrate its efficacy in formulating  partial differential equation models for diverse gradient flow systems. Furthermore,  this principle provides a robust basis for developing numerical schemes that uphold crucial physical properties. Within this framework, several widely used schemes emerge naturally, showing its versatility and applicability. 
\end{abstract}



\section{Introduction}

Physical systems inherently exhibit critical properties such as energy conservation/dissipation relations, mass conservation, and positive densities. Designing numerical schemes that preserve fundamental characteristics is crucial for accurately modeling  and simulating complex systems. {The long tradition of developing structure-preserving numerical methods within the numerical analysis community reflects the importance of maintaining these essential properties in simulation techniques. 
For conservative dynamical systems, a variety of numerical methods have been developed, including symplectic integrators \cite{vogelaere1956methods,ruth1983canonical,feng1984difference} and varitional integrators \cite{marsden2001discrete}, each based on different (but equivalent) formulations of the system. 
Moreover, structure preserving schemes have been extended to encompass more general Lie-Poisson systems \cite{engo2001numerical,celledoni2014introduction}. For a more comprehensive introduction to these methods, see  \cite{marsden2001discrete,hairer2003geometric,feng2010symplectic,morrison2017structure} and the references therein.
}

{Recently, there has been growing interest in developing numerical schemes that ensure energy stability for dissipative systems. Specifically, for dissipative physical systems involving nonconserved unknowns, such as the Allen-Cahn equation, the energy dissipation relation is a crucial property influencing  the stability of numerical schemes. This area of research has attracted considerable attention, particularly cerning energy stability. Previous studies (\cite{Eyre1998, XuTang2006, ShenYY2015}) have established foundational work, while more recent contributions (\cite{yang2016linear, ZYGW2017,SXYSAV2019}) have advanced efficient approaches for ensuring energy stability. 

For dissipative physical systems with conservative unknowns, such as density in the Fokker-Planck equation or in the Planck-Nernst-Poisson systems, maintaining mass conservation, alongside the energy dissipation law,  is essential. Additionally, ensuring the positivity or boundedness of solutions is critical in various applications. Recent advances  have been made in developing numerical methods that preserve these solution properties (see, e.g.,  \cite{LY2014,LW2017,li2020fisher,DJLQ2021,CCWW2022,LiuWWYZ2021}, and references therein). These developments involve a range of techniques, including explicit-implicit time discretizations applied to various PDE reformulations (e.g., \cite{LY12, CCH15,LM21,HS2021,DJLQ2021,LiuWWYZ2021,DCLYZ2021}) and the use  of optimal transport distances (\cite{li2020fisher, CCWW2022, LM23}). Designing higher-order schemes (beyond second order) that preserve all three solution properties remains a challenging task. Ongoing research in higher-order spatial discretizations explores the integration of limiting techniques to ensure  solution bounds (see, e.g., \cite{LY2014, LW2017, SCS18, LWYY22, CLY24} for applications of discontinuous Galerkin methods). This paper presents a fresh perspective rooted in the  Onsager principle, aiming to explore structure-preserving numerical methods for both categories of dissipative physical systems.}


The Onsager principle, a fundamental law in thermodynamic physics \cite{Onsager1931,Onsager1931a}, has played a crucial role as a modeling tool in various soft matter problems \cite{DoiSoft}. In recent years, it has been employed not only as a modeling tool but also as an approximation tool for deriving reduced models \cite{Doi2015,ManXingkun2016,XuXianmin2016}, as well as in the development of numerical methods \cite{lu2021efficient,xu2023variational,xiao2024moving,liu2024variational}, among other applications.
{This effort to enhance the application of the Onsager principle faces two key challenges: (1) Selecting appropriate slow variables for effective application of the Onsager principle; and (2) identifying various free coefficients within  the structured dynamical system.} The present study specifically addresses the utilization of the Onsager principle to derive a natural time discretization for equations of interest.

The Onsager principle extends the Rayleigh’s principle of the least energy dissipation in Stokesian hydrodynamics and is a well-established model for the dynamics of dissipative systems near equilibrium. To elucidate the Onsager principle, we initiate the explanation by considering an ordinary differential system with gradient flow structure. Given the generalized coordinates $y=(y_1, \cdots, y_s)^\top$, their dynamical evolution is described by
the equation 
$$
A\dot y=-\nabla U(y),
$$
where we use the dot for the time derivative, $\dot y=\partial_t y$, $U$ represents a potential function. The matrix $A$ characterizes the energy dissipation in the system and is positive semi-definite, satisfying  $y \cdot A y\geq 0$ for all $y$.  A key insight from Onsager establishes that if the system exhibits microscopic time reversibility, then $A$ is symmetric ($A=A^\top$). This reciprocal relation allows us to express the above time evolution equation as a variational principle. Introducing the Rayleighian
\begin{align}
  R=\Phi+\dot U,   
\end{align}
where $\Phi$ is defined by 
$$
\Phi=\frac{1}{2}\|z\|^2_A:= \frac{1}{2} z\cdot  Az, \quad z=\dot y
$$
and $\dot U$ is defined by
$$
\dot U=\nabla U \cdot \dot y. 
$$
The force balance equation is equivalent to the 
condition
$ \nabla _{\dot y} R= 0$. In other words, the evolution of $y$ is determined by the requirement that $R$ is minimized with respect to $\dot y$. This is the essence of the Onsager principle.

Now, let's employ the same principle to discretize the equation. We straightforwardly adopt a discrete form of $R$:
$$
\Phi(z)+\frac{U(y)-U(y_k)}{\tau},
$$
subject to 
$$
z=\frac{y-y_k}{\tau}.
$$
Here, $\tau$ represents the time step, and $y_k$ is the solution update at time $t_k=k\tau$. With a fixed time step, we can determine $y_{k+1}$ by solving a constraint optimization problem:
\begin{align*}
y_{k+1} & ={\rm argmin}_{y, z}\left\{ 
\Phi(z) + \frac{U(y)-U(y_k)}{\tau}, \quad  y=y_k+\tau z 
\right\} \\
& ={\rm argmin}_{y, z}\left\{ 
U(y)+\tau \Phi(z), \quad y=y_k+\tau z 
\right\} \\
& = {\rm argmin}_{y}\left\{ 
U(y)+\frac{1}{2\tau } \|y-y_k\|^2_A
\right\}.
\end{align*} 
{\color{black} This scheme is known in the literature as 
the De Giorgi's minimizing movement scheme \cite{de1992movimenti,ambrosio2005gradient}.} This scheme exhibits unconditional {energy stability}, meaning that
$$
U(y_{k+1})+ \frac{1}{2\tau } \|y_{k+1}-y_k\|^2_A \leq U(y_k)
$$
holds for any $\tau>0$. 



In this paper, we will apply the above methodology  to a broad class of dissipative physical systems, which might have complicated multi-physics processes. We first show that the Onsager principle can be used to derive the evolution equations for both conserved and non-conserved unknowns. Typical examples include some celebrated nonlinear partial differential equations (PDEs), e.g. the Allen-Cahn equation, the Cahn-Hilliard equation, the Fokker-Planck equations, the Planck-Nernst-Poisson systems, the Maxwell-Stefan diffusion equation, and the relatively new thermodynamically consistent  models for two-phase flow in porous media. {We then consider the time discretization of the dissipative systems.} We consider separately the two categories of problems with conserved or nonconserved unknowns. For problems with nonconserved parameters, the time discrete scheme is similar to that for ODEs. For systems with conserved parameters, we derive novel numerical schemes, which lead to solving an optimization problem in each time step. We show that all the schemes preserve the energy dissipation structure of continuous problems. The mass conservation property is also preserved for conserved parameters. We present the applications of the method to the above mentioned PDEs {with conserved or nonconserved unknowns}. Since we focus on the time discrete schemes, the space discretization and the solution of the optimization problem can be chosen freely. For problems with {nice convexity property, e.g. when the discrete energy functional is convex}, the solution of the optimization problem can be very efficient. 

{Finally, it is important to note that the Onsager principle applies only to purely dissipative systems where inertia effects are negligible. In systems where inertia plays an important role, the Onsager principle cannot be used directly. Instead,  generalized variational forms can be employed, such as the GENERIC approach \cite{grmela1997dynamics,Oettinger_beyond}, the metripletic formulation \cite{materassi2012metriplectic}, the energetic variational  approach \cite{feng2005energetic,hyon2008energetic}, and the generalized Onsager principle \cite{yang2016hydrodynamic,wang2021generalized}, among others \cite{van2020variational,yong2020intrinsic}. These generalized variational forms have also been used to develop numerical schemes in recent studies (e.g., \cite{liu2020variational,wang2022some}).  } 

The structure of the rest of paper is organized as follows. In Section 2, we present the derivation of the evolution equations by using the Onsager principle. In Section 3, we derive the time discrete schemes for two abstract problems, including many important PDEs as specific examples. In Section 4, we discuss briefly the spacial discretization and the optimization methods. We also present some numerical examples to show that our methods can indeed preserve the corresponding physical properties. Some concluding remarks are given in Section 5. 

\section{Onsager principle as a modeling tool}

Consider a {complete} set of variables, denoted as  $\tmmathbf{u} = (u_1, \cdots, u_s)^T$,  that characterize a physical system in $\Omega\in\mathbb{R}^m$. Let $\partial_t{\tmmathbf{u}} = (\partial_t u_1, \cdots, \partial_t u_s)^T$  represent the time derivative of $\tmmathbf{u}$. Assuming the system undergoes an irreversible process within a linear response regime and neglecting inertial effects, the process can be derived by using the the renowned Onsager variational
principle. We demonstrate this for two distinct scenarios: one involving a system without conservation and the other for a mass conserved system.

\subsection{Systems for nonconserved parameters}
Let $\mathcal{E} (\tmmathbf{u})$ represent {the free energy functional} in a system. Its time derivative can be denoted as $\dot{\mathcal{E}} (\tmmathbf{u} ; \partial_t{\tmmathbf{u}})$. According to the Onsager reciprocal relation, the energy dissipation function is {a  positive definite quadratic functional} with respect to $\partial_t{\tmmathbf{u}}$ and can be expressed by $\Phi (\tmmathbf{u} ; \partial_t{\tmmathbf{u}})$, satisfying 
\begin{equation}\label{phi} 
\Phi(\tmmathbf{u};\lambda \tmmathbf{v})=\lambda^2\Phi(\tmmathbf{u}; \tmmathbf{v}), \quad \Phi(\tmmathbf{u}; \tmmathbf{v})\geq 0,
\qquad
\forall \lambda\in \mathbb{R},\ \tmmathbf{v}:\Omega\rightarrow \mathbb{R}^s. 
\end{equation}
The Rayleighian functional is then defined as
\begin{equation}\label{e:Rayleigh1} 
\mathcal{R} (\tmmathbf{u} ; \partial_t{\tmmathbf{u}}) \assign \dot{\mathcal{E}}
   (\tmmathbf{u} ; \partial_t{\tmmathbf{u}}) + \Phi (\tmmathbf{u} ;
   \partial_t{\tmmathbf{u}}). 
\end{equation}
By the Onsager principle, the dynamic equation of the system can be
derived by minimizing the Rayleighean functional with respect to
$\partial_t{\tmmathbf{u}}$, i.e.,  \ \
\begin{equation}\label{e:AbsOnsager}  
\min_{\partial_t{\tmmathbf{u}}}  \mathcal{R} (\tmmathbf{u} ; \partial_t{\tmmathbf{u}}). 
\end{equation}
It is noteworthy that
 $\mathcal{R} (\tmmathbf{u} ; \partial_t{\tmmathbf{u}})$ is a quadratic
positive definite form with respect to $\partial_t{\tmmathbf{u}}$ and
\[ \dot{\mathcal{E}} (\tmmathbf{u} ; \partial_t{\tmmathbf{u}}) = \langle
   \frac{\delta \mathcal{E} (\tmmathbf{u})}{\delta \tmmathbf{u}},
   \partial_t{\tmmathbf{u}} \rangle . \]
{Here $\langle\cdot,\cdot\rangle$ denotes the $L^2$ inner product.}
   This variational problem straightforwardly leads to the dynamic equation:
\begin{equation}\label{e:AbsDynamic}  \frac{\delta \Phi}{\delta (\partial_t{\tmmathbf{u}})} + \frac{\delta \mathcal{E}
   (\tmmathbf{u})}{\delta \tmmathbf{u}} = 0. 
   \end{equation}
{Here $ \frac{\delta \Phi}{\delta (\partial_t{\tmmathbf{u}})}= \frac{\delta \Phi}{\delta \tmmathbf{v}}(\tmmathbf{u};\tmmathbf{v})\Big|_{\tmmathbf{v}=\partial_t{\tmmathbf{u}}}$, or more precisely,  $\langle \frac{\delta \Phi}{\delta (\partial_t{\tmmathbf{u}})}, \delta \tmmathbf{v}\rangle =\frac{d }{d s}\Phi(\tmmathbf{u};\tmmathbf{v}+s \delta \tmmathbf{v})\Big|_{s=0,v=\partial_t{\tmmathbf{u}}}$.}
\subsection{Systems for conserved parameters}
In many systems, certain physical variables are conserved, such as the total mass of each component in a binary system,  { particularly in the absence of phase transitions in the bulk and under normal flux across the boundary}. In such cases, we typically encounter a set of conservation equations:
\begin{equation}\label{e:conserv}
\partial_t {u}_i + \nabla \cdot \tmmathbf{j}_i = 0, \qquad i = 1, \ldots, s. 
\end{equation}
Here $\tmmathbf{j}_i$ represents the flux corresponding to $u_i$. We assume that the problem is defined in a bounded domain $\Omega$, and 
$\tmmathbf{j}_i\cdot\tmmathbf{n}=0$ on $\partial \Omega$, where $\tmmathbf{n}$ is the outward unit normal vector.
In such a system, the
energy dissipation can be defined as a quadratic functional $\Phi
(\tmmathbf{u} ; \tmmathbf{j})$ with respect to $\tmmathbf{j} : =
(\tmmathbf{j}_1^T, \cdots, \tmmathbf{j}_s^T)^T$. Applying the Onsager reciprocal
relation, we assume that $\Phi (\tmmathbf{u}; \tmmathbf{j})$ is a positive
definite quadratic form with respect to $\tmmathbf{j}$, satisfying 
\[
\Phi(\tmmathbf{u};\lambda \tmmathbf{j})=\lambda^2\Phi(\tmmathbf{u}; \tmmathbf{j}), \quad \Phi(\tmmathbf{u};\tmmathbf{j})\geq 0,
\qquad
\forall \lambda\in \mathbb{R},\ \tmmathbf{j}:\Omega\rightarrow \mathbb{R}^{ms}. 
\]
\ The Rayleighean
functional is then given by
\begin{equation}\label{e:Reyleign2}
\mathcal{R} (\tmmathbf{u} ; \partial_t{\tmmathbf{u}}, \tmmathbf{j}) \assign \Phi
   (\tmmathbf{u} ; \tmmathbf{j}) + \dot{\mathcal{E}} (\tmmathbf{u} ;
   \partial_t{\tmmathbf{u}}). 
\end{equation}
By the Onsager principle, the dynamics of the system can be obtained by
minimizing the Rayleighean functional with respect to $(\partial_t{u},
\tmmathbf{j})$ under the constraints of mass conservation equations, i.e.,
\begin{align}
  \min_{ \partial_t{\tmmathbf{u}},\tmmathbf{j}} & \ \mathcal{R} (\tmmathbf{u} ;
  \partial_t{\tmmathbf{u}}, \tmmathbf{j}), \label{e:AbsOnsConserv}\\
  s.t. & \ \partial_t {u}_i + \nabla \cdot \tmmathbf{j}_i = 0, \quad i = 1, \ldots,
  s.\nonumber
\end{align}
{Here and in what follows,  ``$s.t.$" is an abbreviation for ``subject to".}
Once again $\dot{\mathcal{E}} (\tmmathbf{u} ; \partial_t{\tmmathbf{u}})$ is a linear
functional with respect to $\partial_t{\tmmathbf{u}}$. Introduce some multiplies
$\mu_i$, and define
\[ \mathcal{R}_{\mu} \assign \mathcal{R} - \sum_{i = 1}^s \int_{\Omega} \mu_i
   (\partial_t{u}_i + \nabla \cdot \tmmathbf{j}_i) \tmop{d\tmmathbf{x}} . \]
The variational problem reduces to a dynamics equation given by the
Euler-Lagrange equation of $\mathcal{R}_{\mu}$,
\begin{equation}
\left\{
\begin{array}{ll}
    \partial_t{u}_i + \nabla \cdot \tmmathbf{j}_i =0,  & \hspace{3em} i = 1, \ldots,
  s, \\
  \frac{\delta \Phi}{\delta \tmmathbf{j}_i} + \nabla \mu_i  =  0,    & \hspace{3em} i = 1, \ldots, s,\\
   \frac{\delta \mathcal{E}}{\delta u_i} - \mu_i  =  0,&  \hspace{3em} i = 1,
  \ldots, s.
\end{array}
  \right. \label{e:AbsDynamicConserv}
\end{equation}
In the derivation of the second equation, we have used integration by parts and the boundary condition that $\tmmathbf{j}_i = 0$ on $\partial \Omega$. For simplicity, we assume there are no fluxes on the boundary $\partial \Omega$ throughout the paper. Other types of boundary conditions can also be discussed, but they are more involved within the framework of the Onsager principle. We will show some examples below.

\begin{remark}\color{black}
In the derivations above, we begin with free energy and dissipation functions and apply the Onsager principle to derive a dynamic equation. These functions typically reflect  the physical properties of the actual system under consideration.  Alternatively, in certain applications, one might begin with a PDE. In such cases, the energy and dissipation functions can be determined by analyzing the energy dissipation structure of the PDE. However, 
the specific energy and dissipation functions identified may not always be unique (see  \cite{duan2019numerical} for an example with the porous media equation).
\end{remark}
\begin{remark}
{In certain applications, the system is not entirely  dissipative, meaning that the inertial effects cannot be neglected.  In such cases, convective terms appear in the dynamic equations (e.g., the Navier-Stokes equations).  It is important to note that these systems cannot be modeled directly using the Onsager principle. Alternative methods may be employed,  such as the GENERIC method \cite{grmela1997dynamics}, the energetic variational method \cite{feng2005energetic}, and the generalized Onsager principle \cite{wang2021generalized}, among others.}
\end{remark}

\subsection{Application Examples}\label{AE} 
We present several examples widely used in science and engineering. 

{\tmstrong{Example 1}}: The Allen-Cahn equation.

The Allen-Cahn equation (after John W. Cahn and Sam Allen) is a reaction–diffusion equation of mathematical physics which describes the phase transition of a system, characterized by an order parameter $u$, a scalar function defined in a domain $\Omega \subset \mathbb{R}^m$.\quad  The domain $\Omega$ represents the physical space occupied by the system.

In this context, the energy functional and the dissipation function are defined as follows:
\begin{eqnarray}
  \mathcal{E} (u) & = &  \int_{\Omega} \frac{\alpha}{2} | \nabla u |^2 + F (u) \tmop{d\tmmathbf{x}} + \int_{\partial \Omega} \gamma (u) \tmop{dS}, \label{e:EngPF}\\
  \Phi (\partial_t{u}) & = & \frac{\xi_0}{2} \int_{\Omega} | \partial_t{u} |^2 \tmop{d\tmmathbf{x}} +
  \frac{\xi_1 }{2} \int_{\partial \Omega} | \partial_t{u} |^2 \tmop{dS}.
  \label{e:DisspAC}
\end{eqnarray}
Here, $\alpha$ is a positive parameter, $\xi_0, \xi_1$ are parameters, and $F(u)$ represents the bulk free energy density, typically modeled as a double-well function, with a common choice being $F(u) = \frac{(1 - u^2)^2}{4}$. The term $\gamma(u)$ denotes the free energy density on the boundary, assuming the presence of energy dissipation at the system boundary.

Direct calculations yield the following expression:
\begin{eqnarray}
  \dot{\mathcal{E}} (u, \partial_t{u}) & = &  \int_{\Omega} \alpha \nabla u \cdot
  \nabla \partial_t u + F' (u) \partial_t{u} \tmop{d\tmmathbf{x}} + \int_{\partial \Omega} \gamma'  (u) \partial_t{u}
  \tmop{dS}\nonumber\\
  & = & \int_{\Omega} (- \alpha \Delta u + f (u)) \partial_t{u} \tmop{d\tmmathbf{x}} +
  \int_{\partial \Omega} (\alpha \partial_n u + \gamma'  (u)) \partial_t{u} \tmop{dS}. \label{e:DEngPH}
\end{eqnarray}
In the second {line}, we have utilized integration by parts and introduced the notation $f(u) = F'(u)$.

The dynamic equation for the nonconserved coefficient is then reduced to:
\begin{eqnarray}
 \xi_0 \partial_t{u} - \alpha \Delta u + f (u) & = & 0, \quad \tmop{in}\;  \Omega ; \label{e:AC}\\
  \xi_1 \partial_t{u} + \alpha \partial_n u + \gamma'  (u) & = & 0, \quad \tmop{on} \; \partial
  \Omega .\label{e:bndAC}
\end{eqnarray}
When $\xi_1 = 0$, the boundary condition becomes an equilibrium condition, describing the wetting property of the boundary $\partial\Omega$ \cite{lu2021efficient}. If we further assume that $\gamma$ is independent of $u$, then the boundary condition is reduced to the standard homogeneous Neumann condition.

{\tmstrong{Example 2}}: The Cahn-Hilliard equation.

The Cahn-Hilliard equation (after John W. Cahn and John E. Hilliard) is an equation of mathematical physics which describes the process of phase separation in a system, characterized by an order parameter $u$. The free energy functional is identical to that of the Allen-Cahn equation in Equ.~\eqref{e:EngPF}. However, in this case, the parameter $u$ is conserved and satisfies
\[ \partial_t{u} + \nabla \cdot \tmmathbf{j} = 0, \; \text{in}\; \Omega, \quad \tmmathbf{j}\cdot\tmmathbf{n}|_{\partial\Omega}=0,\]
where $\tmmathbf{j} : \Omega \rightarrow \mathbb{R}^m$ is a vector function.
In this scenario, the dissipation functional is defined as
\begin{equation}\label{e:tempPhi}
     \Phi(u;\tmmathbf{j},\partial_t \tmmathbf{u}) = \int_{\Omega} \frac{\tmmathbf{j}^T A (u) \tmmathbf{j}}{2} \tmop{d\tmmathbf{x}}
   + \frac{\xi}{2} \int_{\partial \Omega} | \partial_t{u} |^2 \tmop{dS}, 
\end{equation}
where $A (u)$ is a positive definite symmetric matrix function, akin to a friction coefficient.  {We can define the Reyleighian functional as follows: 
$$\mathcal{R}= \Phi(\tmmathbf{u};\tmmathbf{j},\partial_t \tmmathbf{u})+ \dot{\mathcal{E}} (u, \partial_t{u}),
$$
where $\dot{\mathcal{E}}(\tmmathbf{u}; \partial_t{u})$ is given in Equation~\eqref{e:DEngPH}. We aim to minimize the functional $\mathcal{R}$ with respect to $\tmmathbf{j}$ and $\partial_t \tmmathbf{u}$ under the constraint $\partial_t{u} + \nabla \cdot \tmmathbf{j} = 0, \; \text{in}\; \Omega$. Introducing a Lagrange multiplier $\mu(\tmmathbf{x})$,  we define 
$$
\mathcal{R}_\mu=\mathcal{R}-\int_{\Omega} \mu(\partial_t{u} + \nabla \cdot \tmmathbf{j}) \tmop{d\tmmathbf{x}}.
$$
Setting the first variation of the functional $\mathcal{R}_\mu$ to zero yields the following relations: 
\begin{align*}
   &A(u)\tmmathbf{j} +\nabla\mu=0, \\
    &\mu=-\alpha \Delta u+f(u), \quad \ \ \tmop{in} \  \Omega, \\
    &\xi \partial_t{u}  =  - (\partial_n u + \gamma'  (u)) \quad \ \tmop{on} \ \partial \Omega.
\end{align*}
Combining these results with the conservation equation and the boundary conditions, we obtain } the dynamic equation 
\begin{eqnarray*}
  \partial_t{u} + \nabla \cdot \tmmathbf{j} & = & 0, \hspace{7em} \tmop{in} \ \Omega,\\
  A (u) \tmmathbf{j} & = & - \nabla \mu, \hspace{5em} \ \tmop{in} \ \Omega,\\
  \mu & = & - \alpha \Delta u + f (u), \quad \ \ \tmop{in} \  \Omega, \hspace{3em}\\
  \tmmathbf{j}\cdot\tmmathbf{n} & = & 0 \qquad \hspace{5em} \ \tmop{on} \ \partial \Omega,\\
  \xi \partial_t{u} & = & - (\partial_n u + \gamma'  (u)) \quad \ \tmop{on} \ \partial
  \Omega .
\end{eqnarray*}
Introduce a mobility matrix $M (u) = A^{- 1}$, we obtain 
\begin{eqnarray}
  \partial_t{u} & = & \nabla \cdot (M (u) \nabla \mu), \quad\ 
  \tmop{in} \ \Omega,\label{e:CH_a}\\
  \mu & = & - \alpha \Delta u + f (u), \quad \ \,  \tmop{in} \ \Omega,\label{e:CH_b}\\
 M (u) \nabla \mu \cdot\tmmathbf{n} & = & 0, \qquad \hspace{4em}\ \ \  \tmop{on} \ \partial \Omega,\label{e:bndCH_a}\\
  \xi \partial_t{u} & = & - (\partial_n u + \gamma'  (u)), \quad \tmop{on} \ \partial
  \Omega .\label{e:bndCH_b}
\end{eqnarray}
This represents a general form of the Cahn-Hilliard equation,
which has been used to describe wetting phenomena of a droplet in \cite{chen2014analysis}. Once again, when $\xi =
0$, we obtain the equilibrium boundary condition for $u$ on $\partial \Omega$.

{\tmstrong{Example 3}}: The Fokker-Planck equation.

In statistical mechanics and information theory, the Fokker-Planck equation (FPE) (after Adriaan Fokker and Max Planck in 1914 and 1917) is used to investigate the diffusion of a specific type of particles under the influence of an external potential field. This equation can also be extended to other observables \cite{Le2000}. 
Here we are concerned with FPE having the form 
$$
\partial_t u=\nabla \cdot( \nabla U u)+\beta^{-1}\Delta u,
$$
where $u > 0$ represents  the density function of the particles,  $U(x)$ represents an external potential field, and $\beta>0$ is a given constant. It is known \cite{JKO1998}  the FPE dynamics can be regarded as a gradient flow of the free energy of form  
\[ \mathcal{E} (u) = \beta^{-1} \int_{\Omega} u \log u + u U (\tmmathbf{x}) \tmop{d\tmmathbf{x}}, \]
with respect to the Wasserstein metric on an appropriate class of probability measures. It is also well-known that the Fokker-Planck equation is inherently related to the It$\hat{o}$ stochastic differential equation \cite{Ri1996}
$$
dX(t)=-\nabla U(X(t))dt+\sqrt{2\beta^{-1}}dW(t), \quad X(0)=X^0.
$$
Here, $W(t)$ is a standard $d$-dimensional Wiener process, and $X^0$ is an $d$-dimensional random vector with probability density $u(t=0)$. In this context, $u$ furnishes the probability density at time $t$ for finding the particle at position $x \in \Omega$.  

Here we demonstrate how to derive FPE by the Onsager principle. Considering  the conservation property: 
\[ \partial_t u + \nabla \cdot \tmmathbf{j} = 0, \; \text{in}\; \Omega, \quad \tmmathbf{j}\cdot\tmmathbf{n}|_{\partial\Omega}=0, \]
where $\tmmathbf{j}$ is the mass flux, we define the energy dissipation function as
\[ \Phi = \frac{1}{2} \int_{\Omega} \frac{|\tmmathbf{j}|^2}{u} \tmop{d\tmmathbf{x}} . \]
By direct calculation, the chemical potential in this case is given by
\[ \mu = \frac{\delta \mathcal{E}}{\delta u} = \beta^{-1}(1 + \log u) + U. \]
Then by applying the Onsager principle,
\begin{align*}
\min_{\partial_t{u}, \tmmathbf{j}} &\   \mathcal{R}  = \Phi+\dot{\mathcal{E}}
, 
  \\ 
  s.t. & \ \partial_t{u} + \nabla \cdot \tmmathbf{j} = 0 \; \text{in}\; \Omega, \quad \tmmathbf{j}\cdot\tmmathbf{n}|_{\partial\Omega}=0,
\end{align*}
the dynamic equation for conserved system is reduced to
\begin{eqnarray*}
  \partial_t{u} + \nabla \cdot \tmmathbf{j} & = & 0, \qquad \tmop{in} \ \Omega,\\
  \frac{\tmmathbf{j}}{u} + \nabla \mu & = & 0, \qquad \tmop{in} \
  \Omega,\\
  \mu & = & \beta^{-1}(1 + \log u) + U, \qquad \tmop{in} \ \Omega,\\
  \partial_n \mu & = & 0 \qquad \tmop{on} \  \partial \Omega .
\end{eqnarray*}
We can do further simplification by substituting the third equation into the
second one. This leads to
\begin{eqnarray}
  \partial_t{u} + \nabla \cdot \tmmathbf{j} & = & 0, \qquad \tmop{in} \ \Omega,
  \label{e:FP_a}\\
  \tmmathbf{j} & = & - (\beta^{-1} \nabla u + u \nabla U), \qquad \tmop{in} \
  \Omega,\label{e:FP_b}\\
  \tmmathbf{j} \cdot \tmmathbf{n} & = & 0 \qquad \tmop{on} \; \partial \Omega .
  \label{e:bndFP}
\end{eqnarray}

{\tmstrong{Example 4}}: The Planck-Nernst-Poisson equation.

The Nernst–Planck equation, named after Walther Nernst and Max Planck, is a conservation of mass equation employed to depict the movement of a charged chemical species in a fluid medium. It extends Fick's law of diffusion to account for cases where the diffusing particles are also influenced by electrostatic forces, as governed by the Poisson equation. The Planck-Nernst-Poisson (PNP) equation specifically describes the diffusion of ions in solutions. It is crucial to consider the energy associated with the static electric field.

Suppose there are $s$ types  of ions, and the density of particles is characterized by
$u_i$ for $i=1,\ldots, s$.  These density functions satisfy the conservation equation:
\[ \partial_t{u}_i + \nabla \cdot \tmmathbf{j}_i = 0 \; \text{in}\; \Omega, \quad i = 1, \ldots, s; \quad \tmmathbf{j}_i \cdot \tmmathbf{n}=0 \quad \text{on}\; \partial \Omega. \]
The total free energy is given by 
\[ \mathcal{E} (\tmmathbf{u}) = \int_{\Omega} \sum^s_{i = 1} u_i \log u_i +
   \frac{\varepsilon_0}{2} | \nabla \varphi (\tmmathbf{x}) |^2 \tmop{d\tmmathbf{x}}, \]
where $\varphi$ is the electric potential, $\varepsilon_0(x)$ is the permittivity. It is related to $u_i$ by the
static electric field equation:
\begin{equation}\label{e:electPoten}
 - \nabla \cdot (\varepsilon_0 \nabla \varphi) = f(\tmmathbf{x})+ \sum_{i = 1}^n z_i u_i, \; \text{in}\ \Omega,
\end{equation}
where $z_i$ is the rescaled charge, 
$f(x)$ is the permanent (fixed) charge density of the system. Boundary conditions for $\varphi$ can vary, we simply take $\partial_n \varphi$ on $\partial \Omega$.
The energy dissipation functions is defined as,
\[ \Phi \assign \frac{1}{2} \int_{\Omega} \sum_i D_i^{-1}(\tmmathbf{x}) u_i^{- 1} | \tmmathbf{j}_i
   |^2 \tmop{d\tmmathbf{x}}, \]
 where $D_i(x)>0$ is an diffusion coefficient for $i$-th ion.
 Direct calculations show that
\[ \frac{\delta \mathcal{E}}{\delta u_i} = (1 + \log u_i + z_i \varphi), \]
where we have used the static electric field equation   and the boundary
condition $\partial_n \varphi = 0$. Then by applying the Onsager principle,
\begin{align*}
\min_{\partial_t{\tmmathbf{u}}, \tmmathbf{j}} &\   \mathcal{R} = \frac{1}{2}
  \int_{\Omega} \sum_i D_i^{-1}(x) u_i^{- 1} | \tmmathbf{j}_i |^2 \tmop{d\tmmathbf{x}} +
  \dot{\mathcal{E}},\\
  s.t. & \ \partial_t{u}_i + \nabla \cdot \tmmathbf{j}_i = 0 \; \text{in}\; \Omega, \quad \tmmathbf{j}_i\cdot\tmmathbf{n}|_{\partial\Omega}=0,
\end{align*}
we obtain the PNP system of  equations: 
\begin{align}
& \partial_t u_i   = \nabla\cdot [D_i(\tmmathbf{x}) (\nabla u_i +  z_i u_i
  \nabla \varphi)], \quad \tmmathbf{x}\in \Omega,\quad i = 1, \ldots s ;\label{e:PNP_a}\\
  & - \nabla \cdot (\varepsilon_0 \nabla \varphi) = f(\tmmathbf{x})+ \sum_{i = 1}^s z_i u_i, \qquad \tmmathbf{x}\in \Omega,\label{e:PNP_b}\\
  & \partial_n \varphi  =0 \quad \text{on}\; \partial \Omega. \label{e:bndPNP}
\end{align}
Note that the external electrostatic potential $\varphi$ is influenced by applied potential, which can be modeled by prescribing a boundary
condition. The analysis above applies well to a general form of boundary conditions:
$$
\alpha \varphi +\beta \epsilon_0(\tmmathbf{x})\partial_n \varphi  =0, \quad \tmmathbf{x}\in \partial \Omega,
$$
if we take a modified energy of form 
\[ \mathcal{E} (\tmmathbf{u}) = \int_{\Omega} \left( \sum^s_{i = 1} u_i \log u_i +
   \frac{1}{2} (f+\sum_{i=1}^s z_iu_i) \varphi(\tmmathbf{x}) \right) \tmop{d\tmmathbf{x}}.
 \]
Here $\alpha, \beta$ are physical parameters such that $\alpha\cdot \beta \geq 0$. Refer to \cite{LM23} for cases with non-homogeneous boundary conditions. 

{\tmstrong{Example 5}}: The Maxwell-Stefan diffusion equation.

For a system containing multiple types of particles, the Maxwell-Stefan diffusion equation is employed to describe the diffusion of the multi-component system. The equations that describe these transport processes have been developed independently and in parallel by James Clerk Maxwell (1965) for dilute gases and Josef Stefan (1871) for liquids.  Let $u_i > 0, i = 1, \ldots, s$ denote the number density of each component, and all components are conserved, satisfying the conservation equation
\begin{equation}\label{e:conservMS} \partial_t u_i + \nabla \cdot \tmmathbf{j}_i = 0 \; \; \text{in}\; \Omega,\quad j = 1, \ldots, s,
\quad \tmmathbf{j}_i\cdot\tmmathbf{n}|_{\partial\Omega}=0.
\end{equation}
Here, $\tmmathbf{j}_i$ is the flux for the $i$-th component. The total number density of all components is assumed to be constant:
\begin{equation}\label{e:AllconservMS} 
\sum^s_{i = 1} u_i = 1. 
\end{equation}
The free energy in the system is given by
\begin{equation}\label{e:EngMS}
    \mathcal{E} (\tmmathbf{u}) = \int_{\Omega} \sum_{i = 1}^s u_i \log u_i
   \tmop{d\tmmathbf{x}}, 
\end{equation} 
and the dissipation function is defined as
\begin{equation}\label{e:DisspMS}
\Phi (\tmmathbf{u};\tmmathbf{j}) = \frac{1}{4} \int_{\Omega} \sum_{i, j = 1}^s b_{i j}
   u_i u_j \left| \frac{\tmmathbf{j}_i}{u_i} - \frac{\tmmathbf{j}_j}{u_j}
   \right|^2 \tmop{d\tmmathbf{x}}. 
   \end{equation}
Here coefficient $B = (b_{i j})$ is symmetric and positive definite.

To simplify the derivation, we introduce new variables $\tmmathbf v_i = \frac{\tmmathbf j_i}{u_i}$ to replace $\tmmathbf j_i$. In this case, the dissipation function in Equ.~\eqref{e:DisspMS} is rewritten as
\[ \Phi (\tmmathbf{u} ; \tmmathbf{v}) = \frac{1}{4} \int_{\Omega} \sum_{i, j
   = 1}^s b_{i j} u_i u_j | \tmmathbf{v}_i - \tmmathbf{v}_j |^2 \tmop{d\tmmathbf{x}} . \]
The mass conservation equation~\eqref{e:conservMS} is reduced to
\[ \partial_t u_i + \nabla \cdot (u_i  \tmmathbf{v}_i) = 0 \; \text{in}\; \Omega,\quad j = 1,
   \ldots, s; \quad \tmmathbf{u}\cdot\tmmathbf{n}|_{\partial\Omega}=0. \]
   The additional constraint in Equ.~\eqref{e:AllconservMS} is equivalent to the equation
\[ \sum_{i = 1}^s \nabla \cdot (u_i \tmmathbf{v}_i) = 0. \]
Applying the Onsager variational principle, direct calculations lead to
\[ \dot{\mathcal{E}} (\tmmathbf{u} ; \partial_t {\tmmathbf{u}}) = \int_{\Omega}
   \sum_{i = 1}^s (1 + \log u_i) \partial_t u_i \tmop{d\tmmathbf{x}} . \]
Then the Rayleighean is defined as
\[ \mathcal{R} (\tmmathbf{u} ; \partial_t {\tmmathbf{u}}, \tmmathbf{v}) =
   \frac{1}{4} \int_{\Omega} \sum_{i, j = 1}^s b_{i j} u_i u_j |
   \tmmathbf{v}_i - \tmmathbf{v}_j |^2 \tmop{d\tmmathbf{x}} + \int_{\Omega} \sum_{i =
   1}^s (1 + \log u_i) \partial_t u_i \tmop{d\tmmathbf{x}} . \]
   The dynamic equation is derived by minimizing the Rayleighian functional:
\begin{align*}
\min_{\dot{\tmmathbf{u}},\tmmathbf{v}} & \  \mathcal{R} (\tmmathbf{u} ;
  \partial_t{\tmmathbf{u}}, \tmmathbf{v}),\\
  s.t. & \ \partial_t u_i + \nabla \cdot (u_i  \tmmathbf{v}_i) = 0, \qquad j
  = 1, \ldots, s,\\
  & \ \sum_{i = 1}^s \nabla \cdot (u_i \tmmathbf{v}_i) = 0.
\end{align*}
Introducing Lagrangian multipliers $\mu_i$ for the conservation equation and a multiplier $p$ for the last equation, we set
\[ \mathcal{R}_{\mu} (\tmmathbf{u} ; \partial_t{\tmmathbf{u}}, \tmmathbf{v}) \assign
   \mathcal{R}(\tmmathbf{u} ; \partial_t{\tmmathbf{u}}, \tmmathbf{v}) -
   \sum_{i = 1}^s \int_{\Omega} \mu_i (\partial_t u_i + \nabla \cdot (u_i 
   \tmmathbf{v}_i)) \tmop{d\tmmathbf{x}} - \int_{\Omega} p \left( \sum_{i = 1}^s \nabla
   \cdot (u_i \tmmathbf{v}_i) \right) \tmop{d\tmmathbf{x}} . \]
Considering the first order derivation of the functional, we obtain the system: 
\begin{eqnarray}
  \partial_t u_i + \nabla \cdot (u_i  \tmmathbf{v}_i) & = & 0, \qquad
  \tmop{in} \ \Omega,\label{e:MS0_a}\\
  \sum_{j = 1}^s b_{i j} u_j  (\tmmathbf{v}_i - \tmmathbf{v}_j) + \nabla \mu_i
  & = & - \nabla p, \qquad \tmop{in}  \ \Omega,\label{e:MS0_b}\\
  \mu_i & = & 1 + \log u_i \qquad \tmop{in} \ \Omega,\label{e:MS0_c}\\
  \sum_{i = 1}^s \nabla \cdot (u_i \tmmathbf{v}_i) & = & 0, \qquad \tmop{in}\ 
  \Omega,\label{e:bndMS0}\
\end{eqnarray}
where we have used the boundary conditions $\partial_n \mu_i = 0$ and
$\partial_n p = 0$. Substitute the third equation into the second one, we get 
\[ \sum_{j = 1}^s b_{i j} u_j  (\tmmathbf{v}_i - \tmmathbf{v}_j) +
   \frac{1}{u_i} \nabla u_i = - \nabla p. \]
By the symmetricity of $b_{i j}$, we have $\sum_{j = 1}^s b_{i j} u_i u_j 
(\tmmathbf{v}_i - \tmmathbf{v}_j) = 0$. Thus by multiplying the above equation
with $u_i$ and do summation with respect to $i$, we obtain 
\begin{equation}
     - \nabla p = \frac{1}{\sum_{i = 1}^s u_i} \sum_{i = 1}^s \nabla u_i . \label{e:tempP}
\end{equation}
The system~\eqref{e:MS0_a}-\eqref{e:bndMS0} is further simplified to
\begin{eqnarray*}
  \partial_t u_i + \nabla \cdot (u_i  \tmmathbf{v}_i) & = & 0, \qquad
  \tmop{in} \ \Omega,\\
  \sum_{j = 1}^s b_{i j} u_j  (\tmmathbf{v}_i - \tmmathbf{v}_j) +
  \frac{1}{u_i} \nabla u_i & = & \frac{1}{\sum_{j = 1}^s u_j} \sum_{j = 1}^s
  \nabla u_j, \qquad \tmop{in} \ \Omega .
\end{eqnarray*}
Note that the boundary conditions are reduced to $\partial_n u_i = 0$ on
$\partial \Omega$. The system above can be expressed as
\begin{eqnarray}
  \partial_t u_i + \nabla \cdot (u_i  \tmmathbf{v}_i) & = & 0, \qquad
  \tmop{in} \ \Omega,\label{e:MS_a}\\
  \sum_{j = 1}^s b_{i j} u_j  (\tmmathbf{v}_i - \tmmathbf{v}_j) + \nabla \log
  u_i & = & \frac{1}{\sum_{j = 1}^s u_j} \sum_{j = 1}^s u_j \nabla \log u_j,
  \quad \tmop{in}  \ \Omega,\label{e:MS_b}\\
  \partial_n u_i & = & 0 \qquad \tmop{on} \  \partial \Omega.\label{e:bndMS}
\end{eqnarray}

{\tmstrong{Example 6}}: The incompressible and immiscible multi-phase flow in porous media.

We now turn our attention to the incompressible and immiscible multi-phase flow in porous media which has extensive applications in hydrology and petroleum reservoir engineerings. Recently, a thermodynamically consistent model was developed for the incompressible and immiscible two-phase flow in porous media \cite{GaoKouSW}. Different from the classical models for two-phase flow in porous media, the  model introduces a logarithmic free energy to characterize the capillarity effect, and the system satisfies {an energy dissipation relation shown in \cite{GaoKouSW}}. In the following, we can also rebuild the thermodynamically consistent model for the incompressible and immiscible multi-phase flow in porous media based on the Onsager principle. 

Let $u_i, i = 1, \ldots, s,$
represent the volume fraction of the $i$-th phase { which is also known as saturation}, ensuring  that $\sum_{i = 1}^s u_i = 1$.
Let $\phi$ be porosity of the porous media. The unknown function satisfies the
following conservation law: 
\[ \phi \partial_t {u}_i + \nabla \cdot \tmmathbf{v}_i = 0, \qquad i = 1, \ldots, s,\qquad \tmmathbf{v}_i\cdot\tmmathbf{n}|_{\partial\Omega}=0.
\]
Here, $\tmmathbf{v}_i$ denotes the average velocity of the $i$-th phase fluid.
In some cases, assuming that the equilibrium state is determined
by a free energy $\mathcal{E} (\tmmathbf{u}) = \int_{\Omega} \phi F
(\tmmathbf{u}) d x$  is convenient.  One possible choice for $F (\tmmathbf{u})$   is  \[F (\tmmathbf{u}) = \sum_{j = 1}^s \sigma_i u_j
( \log u_j -1)+ \sum_{i, j = 1}^s \alpha_{i j} u_i u_j + \sum_{j = 1}^s b_j
u_j.\] The energy dissipation in this case is given by
\[ \Phi = \int_{\Omega} \sum_{i = 1}^s \frac{1}{2} \tmmathbf{v}_i K_i^{- 1}
   \tmmathbf{v}_i \tmop{d\tmmathbf{x}}. \]
Here,  $K_i$ is a positive definite and symmetric matrix which is dependent of $\tmmathbf{u}$. Similar to the derivation for the Maxwell-Stefan equation, we obtain 
\begin{eqnarray}
  \phi \partial_t {u}_i + \nabla \cdot \tmmathbf{v}_i & = & 0, \quad {i=1,\ldots,s}, \quad \tmop{in} \
  \Omega,\label{e:Porous_a}\\
  \tmmathbf{v_i} & = & - K_i (\nabla \mu_i + \nabla p),\quad {i=1,\ldots,s},  \quad \tmop{in} \
  \Omega,\label{e:Porous_b}\\
  \mu_i & = & \frac{\partial F}{\partial u_i},\quad {i=1,\ldots,s},  \quad \tmop{in} \ \Omega,\label{e:Porous_c}\\
  \sum_{i = 1}^s u_i & = & 1, \qquad \tmop{in} \ \Omega,\label{e:Porous_d}\\
  \partial_n \mu_i + \partial_n p & = & 0, \qquad \tmop{on} \ \partial \Omega .\label{e:bndPorous}
\end{eqnarray}

{\tmstrong{Example 7}}: The multi-phase flow in porous media with rock compressibility.

For the multi-phase flow in porous media with rock compressibility, the variation of porosity with respect to effective pressure can be expressed as
\begin{eqnarray}
\frac{d \phi}{\phi} = \gamma d p_e,\label{compress_phi}
\end{eqnarray}
which yields 
\begin{eqnarray}\label{phi_pe}
\phi=\phi_r e^{\gamma(p_e-p_r)}.
\end{eqnarray}
Here $\gamma$ is the rock compressibility coefficient, $p_r$ is the reference or initial pressure and $\phi_r$ is the porosity at the reference pressure. The absolute permeability $K$ of the porous media changes with the porosity according to the Kozeny-Carman equation:
\begin{eqnarray}\label{equ_permeability}
   K =   K_0 \frac{\phi^3(1-\phi_r)^2}{\phi_r^3(1-\phi)^2},
\end{eqnarray}
where $K_0$ is the initial intrinsic permeability. The rock compressibility caused by the pore fluid pressure has been recognized as an important factor influencing many subsurface processes which include the oil/gas production and the geological stability. For the modeling of the changes of rock properties, one approach is to use the Biot-type model for the rock, and another one is the rock compressibility model as (\ref{compress_phi}). A thermodynamically consistent model for the incompressible and immiscible two-phase flow in porous media with rock compressibility was developed in \cite{KouWChenS}. We can rebuild the thermodynamically consistent model for the multi-phase case by the Onsager principle in the following.

Now, let's introduce the rock free energy denoted by $R$. We assume that the work done by the effective pore fluid pressure exerted on rocks is transferred to the rock free energy. The variation of rock free energy with respect to effective pressure is described as:
\begin{eqnarray*}
  dR= p_e  d \phi   .
\end{eqnarray*}
The total free energy is given by  $$
\mathcal{E} (\tmmathbf{u},\phi) = \int_{\Omega} \phi F
(\tmmathbf{u}) \tmop{d\tmmathbf{x}} + \int_\Omega R \tmop{d\tmmathbf{x}},$$
where $F(\tmmathbf{u})$ is given as Example 6. The time derivative of the energy, $\dot{\mathcal{E}} (\tmmathbf{u};\partial_t{\tmmathbf{u}},\phi;\partial_t{\phi})$, is expressed as
\[
\dot{\mathcal{E}} (\tmmathbf{u};\partial_t{\tmmathbf{u}},\phi;\partial_t{\phi}) = \int_{\Omega} \partial_t ( \phi  F
(\tmmathbf{u}) ) \tmop{d\tmmathbf{x}} + \int_\Omega p_e \partial_t{\phi} \tmop{d\tmmathbf{x}}.
\]
The energy dissipation is similar to that in Example 6.

The unknown functions satisfy the conservation laws:
\[
\partial_t (\phi u_i) + \nabla \cdot \tmmathbf{v}_i = 0, \quad i = 1, \ldots, s.
\]
The saturation of each phase satisfies the saturation constraint $\sum^s_{i=1}u_i=1$. By the conservation law, we also have
\[
\sum^s_{i=1}u_i \partial_t{\phi} + \sum^s_{i=1} \nabla \cdot \tmmathbf{v}_i = 0.
\]
If the above equation holds true, we can directly obtain the saturation constraint  $\sum^s_{i=1}u_i=1$ by the conservation law and the initial saturation constraint $\sum^s_{i=1}u_i (0,x)=1$.

The dynamic equation can be derived by minimizing the Rayleighian functional as follows:
\begin{eqnarray*}
  \min_{\partial_t{\tmmathbf{u}},  \partial_t{\phi}, \tmmathbf{v}} &  & \mathcal{R} (\tmmathbf{u} ;
  \partial_t{\tmmathbf{u}}, \phi; \partial_t{\phi}, \tmmathbf{v})\\
  s.t. & & \partial_t (\phi u_i) + \nabla \cdot \tmmathbf{v}_i = 0, \quad i = 1, \ldots, s, \\
 & & \sum^s_{i=1}u_i \partial_t{\phi} + \sum^s_{i=1} \nabla \cdot\tmmathbf{v}_i = 0.
\end{eqnarray*}
Introducing $\mu = (\mu_1,\ldots,\mu_s)$ and $p$ as the Lagrangian multipliers, we set
\begin{eqnarray*}
\mathcal{R}_L (\tmmathbf{u} ;
  \partial_t{\tmmathbf{u}}, \phi; \partial_t{\phi}, \tmmathbf{v},\mu,p) &=&  \mathcal{R}(\tmmathbf{u} ;
  \partial_t{\tmmathbf{u}}, \phi; \partial_t{\phi}, \tmmathbf{v})  - \sum^s_{i=1}\int_\Omega \mu_i(\partial_t(\phi u_i) + \nabla \cdot \tmmathbf{v}_i ) \tmop{d\tmmathbf{x}} \\
  && -  \int_\Omega p \left( \sum^s_{i=1}u_i \partial_t{\phi} + \sum^s_{i=1} \nabla \cdot\tmmathbf{v}_i  \right) \tmop{d\tmmathbf{x}}.
\end{eqnarray*}
By the Euler-Lagrange equation, (\ref{phi_pe}) and (\ref{equ_permeability}), we have 
\begin{eqnarray}
  \partial_t (\phi u_i)+ \nabla \cdot \tmmathbf{v}_i & = & 0, \quad {i=1,\ldots,s}, \quad \tmop{in} \ 
  \Omega, \label{pm_rock_1}\\
  \tmmathbf{v_i} & = & - K_i (\nabla \mu_i + \nabla p),\quad {i=1,\ldots,s},  \quad \tmop{in} \ 
  \Omega,\\
  \mu_i & = & \frac{\partial F}{\partial u_i}, \quad {i=1,\ldots,s}, \quad \tmop{in} \ \Omega,\\
  \sum_{i = 1}^s u_i & = & 1, \qquad \tmop{in} 
  \ \Omega,\\
  p_e & = & p + \sum^s_{i=1} u_i \mu_i - F(\tmmathbf{u}), \quad \tmop{in} \ 
  \Omega,\\
  \phi &=& \phi_r e^{\gamma(p_e-p_r)},\quad \tmop{in} \ 
  \Omega,\\
  K &=&   K_0 \frac{\phi^3(1-\phi_r)^2}{\phi_r^3(1-\phi)^2}, \quad \tmop{in} \ \Omega.\label{pm_rock_7}
\end{eqnarray}
Here $K_i = \lambda_i K$, and $\lambda_i = \frac{k_{ri}(\tmmathbf{u})}{\eta_i}$, where $k_{ri}$ and $\eta_i$ are the relative permeability and viscosity of the phase $i$.

\section{Time-discretization  based on the Onsager principle}
In this section, we present the time discretization of partial differential equations using the Onsager principle. We first introduce the main idea for the abstract problems and then apply them to some typical examples.
  
We discretize the time interval $[0, T]$ by considering $0 = t_0 < t_1 <
\cdots < t_N = T$. Suppose that we already computed the solution at time
$t_k$. We will compute $\tmmathbf{u}^{k + 1}$ by employing a discrete version of
the Onsager principle.

\subsection{Discretization for systems with nonconserved parameters}

Suppose $\tmmathbf{u}^k$ is known with a time step $\tau = t_{k + 1} - t_k$. We discretize
the Rayleighean functional as follows,
\begin{equation} \label{e:discRayleign0}
\mathcal{R}_{\tau}^k (\tmmathbf{u}^k ; \tmmathbf{u}) \assign D_{\tau}
   \mathcal{E} (\tmmathbf{u}^k ; \tmmathbf{u}) + \Phi^k_{\tau} (\tmmathbf{u}^k,
   D_{\tau} \tmmathbf{u}^k), 
\end{equation}
where
\begin{eqnarray}
  \Phi^k_{\tau} (\tmmathbf{u}, D_{\tau} \tmmathbf{u}^k) & \assign & \Phi
  \left( \tmmathbf{u}^k ; \frac{\tmmathbf{u} - \tmmathbf{u}^k}{\tau} \right),
  \label{e:discDissp0}\\
  D_{\tau} \mathcal{E} (\tmmathbf{u}^k ; \tmmathbf{u}) & \assign &
  \frac{\mathcal{E} (\tmmathbf{u}) - \mathcal{E} (\tmmathbf{u}^k)}{\tau} .
   \label{e:discEng0}
\end{eqnarray}
Then the unknowns $\tmmathbf{u}^{k + 1}$ at time $t_{k + 1}$ are computed by
solving the minimization problem,
\[ \tmmathbf{u}^{k + 1} = \tmop{argmin}_{\tmmathbf{u}} \mathcal{R}_{\tau}^k
   (\tmmathbf{u}^k ; \tmmathbf{u}). 
   \]
Or equivalently,
\begin{equation} \label{e:JKO}
\tmmathbf{u}^{k + 1} = \tmop{argmin}_{\tmmathbf{u}} \mathcal{E}
   (\tmmathbf{u}) + \frac{1}{\tau} \Phi (\tmmathbf{u}^k ; \tmmathbf{u} -
   \tmmathbf{u}^k) . 
\end{equation}
Here we have used the fact that $\Phi (\tmmathbf{u} ; \partial_t {\tmmathbf{u}})$ is
a positive definite quadratic form of $\partial_t{\tmmathbf{u}}$. We can see that
$\tmmathbf{u}^{k + 1}$ can be seen as a generalized minimizing movement
solution. It is also referred to as a JKO scheme in the literature \cite{JKO1998}.

The scheme is unconditionally stable by definition. The following {claim}
is easy to verify. Denoted by
$J_{\tau}$ the set of minimizers of the above problem~\eqref{e:JKO}.
\begin{claim}
  When $J_{\tau}$ is not empty, for any choice $u^{k + 1} \in J_{\tau}$, we
  have that
  \[ \mathcal{E} (\tmmathbf{u}^{k + 1}) \leq \mathcal{E} (\tmmathbf{u}^k) -
     \frac{1}{\tau} \Phi (\tmmathbf{u}^{k} ; \tmmathbf{u}^{k + 1} -
     \tmmathbf{u}^k) \leq \mathcal{E} (\tmmathbf{u}^k) . \]
\end{claim}

\begin{proof}
  By definition, we have
  \[ \mathcal{E} (\tmmathbf{u}^{k + 1}) + \frac{1}{\tau} \Phi (\tmmathbf{u}^{k
     } ; \tmmathbf{u}^{k + 1} - \tmmathbf{u}^k) \leq \mathcal{E}
     (\tmmathbf{u}^k) + \frac{1}{\tau} \Phi (\tmmathbf{u}^k ; \tmmathbf{0}) =
     \mathcal{E} (\tmmathbf{u}^k) . \]
  Here we have used the fact that $\Phi (\tmmathbf{u} ; \tmmathbf{j})$ is a
  positive definite quadratic form with respect to $\tmmathbf{j}$.
  
  \ 
\end{proof}

\subsection{Discretization for systems with conserved parameters \ }
When the parameters in a physical system are conserved, we can discretize the
system similarly. Suppose, again,  that $\tmmathbf{u}^k$ is already known and we will compute
$\tmmathbf{u}^{k + 1}$ by using the discrete Onsager principle. We first
discretize the conservation equations as follows:
\begin{equation}\label{e:distConserv}
\frac{u_i - u_i^k}{\tau} + \nabla \cdot \tmmathbf{j}_i = 0, \end{equation}
where $\tmmathbf{j}_i$ is independent of time and satisfies $\tmmathbf{j}_i
\cdot \tmmathbf{n} = 0$ on $\partial \Omega$. Then,  the energy functional and the
dissipation function are discretized as follows:
\begin{eqnarray*}
  \Phi^k_{\tau} (\tmmathbf{u}^k, \tmmathbf{j}) & \assign & \Phi (\tmmathbf{u}^k ;
  \tmmathbf{j}),\\
  D_{\tau} \mathcal{E} (\tmmathbf{u}^k ; \tmmathbf{u}) & \assign &
  \frac{\mathcal{E} (\tmmathbf{u}) - \mathcal{E} (\tmmathbf{u}^k)}{\tau} .
\end{eqnarray*}
The $\tmmathbf{u}^{k + 1}$ and $\tmmathbf{j}^{k + 1}$ are obtained by
minimizing the discrete Rayleighean functional as follows,
\begin{eqnarray}
  (\tmmathbf{u}^{k + 1}, \tmmathbf{j}^{k + 1}) = \tmop{argmin}_{\tmmathbf{u},
  \tmmathbf{j}} \mathcal{R}_{\tau}^k (\tmmathbf{u}, \tmmathbf{j}) & \assign &
  \Phi^k_{\tau} (\tmmathbf{u}^k, \tmmathbf{j}) + D_{\tau} \mathcal{E}
  (\tmmathbf{u}^k ; \tmmathbf{u}) \label{e:discOnsCsv1}\\
  s.t. &  & \frac{u_i - u_i^k}{\tau} + \nabla \cdot \tmmathbf{j}_i = 0, \quad
  i = 1, \ldots, s. \nonumber
\end{eqnarray}
By introduce a variable $\tmmathbf{m} = \tau \tmmathbf{j},$
the problem is equivalent to
\begin{align}
  (\tmmathbf{u}^{k + 1}, \tmmathbf{m}^{k + 1}) = \tmop{argmin}_{\tmmathbf{u},
  \tmmathbf{m}} &  \mathcal{E} (\tmmathbf{u}) + \frac{1}{\tau} \Phi^k_{\tau}
  (\tmmathbf{u}^k, \tmmathbf{m}),\label{e:discOnsCsv2}\\
  s.t. &\ u_i - u_i^k + \nabla \cdot \tmmathbf{m}_i = 0.\nonumber
\end{align}
Here we again used the fact that $\Phi (\tmmathbf{u}^k ; \tmmathbf{j})$ is a
quadratic form with respect to $\tmmathbf{j}$.

\begin{proposition}
  Suppose $(\tmmathbf{u}^{k + 1}, \tmmathbf{m}^{k + 1})$ is a minimizer of the
  minimizing problem~\eqref{e:discOnsCsv2}, we then have
  \[ \mathcal{E} (\tmmathbf{u}^{k + 1}) \leq \mathcal{E} (\tmmathbf{u}^k) -
     \frac{1}{\tau} \Phi (\tmmathbf{u}^{k} ; \tmmathbf{m}^{k + 1}) \leq
     \mathcal{E} (\tmmathbf{u}^k), \]
  and $\tmmathbf{u}^{k + 1} $satisfies the mass conservation equation that
  \[ \int_{\Omega} u_i^{k + 1} \tmop{d\tmmathbf{x}} = \int_{\Omega} u_i^k \tmop{d\tmmathbf{x}},
     \quad i = 1, \ldots s. \]
\end{proposition}

\begin{proof}
  The proof of the energy inequality is similar to that in Proposition 1. By
  definition, we have
  \[ \mathcal{E} (\tmmathbf{u}^{k + 1}) + \frac{1}{\tau} \Phi (\tmmathbf{u}^{k
     } ; \tmmathbf{m}^{k + 1}) \leq \mathcal{E} (\tmmathbf{u}^k) +
     \frac{1}{\tau} \Phi (\tmmathbf{u}^k ; \tmmathbf{0}) = \mathcal{E}
     (\tmmathbf{u}^k) . \]
  Here we have used the fact that $\Phi (\tmmathbf{u} ; \tmmathbf{j})$ is a
  positive definite quadratic form with respect to $\tmmathbf{j}$. The mass
  conservation equation is obtained by integrate the constraint
  \[ \int_{\Omega} u_i^{k + 1} \tmop{d\tmmathbf{x}} = \int_{\Omega} u_i^k \tmop{d\tmmathbf{x}} +
     \int_{\Omega} \nabla \cdot \tmmathbf{m}^{k + 1} \tmop{d\tmmathbf{x}} = \int_{\Omega}
     u_i^k \tmop{d\tmmathbf{x}} + \int_{\partial \Omega} \tmmathbf{m}^{k + 1} \cdot
     \tmmathbf{n} \tmop{d\tmmathbf{x}} = \int_{\Omega} u_i^k \tmop{d\tmmathbf{x}}, \]
  where we have used the assumption that $\tmmathbf{} \tmmathbf{m}^{k + 1}
  \cdot \tmmathbf{n} = \tau \tmmathbf{j}^{k + 1} \cdot \tmmathbf{n}=0 .$
  
  \ 
\end{proof}

\begin{remark}
Notice that we derive the abstract discrete schemes under the no flux boundary condition. For more general conditions, 
we can also derive the corresponding schemes in a similar approach.
For example, if we consider periodic conditions, all the above derivations are same as that for no flux boundary conditions.
\end{remark}
\begin{remark}
{We believe that the present method can be generalized to higher orders  by using higher order approximation (e.g., the k-th order BDF scheme) for the dissipation functional $\Phi$ and the energy change rate $\dot{\mathcal{E}}$. However, this direct generalization may compromise the desirable property of energy stability. A more refined approach to  generalizing this work to high orders will be addressed in future research.}
\end{remark}
\subsection{Time discretization by examples}

{\tmstrong{Example 1}}: The Allen-Cahn equation.

Consider the Allen-Cahn equation described in Example 1 of  Section \ref{AE}, we apply the Onsager principle to obtain the following time discretization: 
\begin{eqnarray*}
  \tmmathbf{u}^{k + 1} & = & \tmop{argmin}_{u\in H^1(\Omega)} \left\{ \mathcal{E}(u) +\frac{1}{\tau}
  \Phi(u-u^k)\right\},\\
\end{eqnarray*}
where \begin{eqnarray*}
  \mathcal{E} (u) & = &  \int_{\Omega} \frac{\alpha}{2} | \nabla u |^2 + F (u) \tmop{d\tmmathbf{x}} + \int_{\partial \Omega} \gamma (u) \tmop{dS},\\
  \Phi (v) & = & \frac{1}{2} \int_{\Omega} |v|^2 \tmop{d\tmmathbf{x}} +
  \frac{\xi_1 }{2} \int_{\partial \Omega} |v|^2 \tmop{dS}.
\end{eqnarray*}
This is a constraint minimization problem.
When $\tau$ is small, the minimizing problem may have a unique minimizer,
which can be obtained by the backward Euler scheme: 
\begin{eqnarray*}
  & \frac{u^{k + 1} - u^k}{\tau}  =  \alpha \Delta u^{k + 1} -
  F'(u^{k + 1}) \quad \tmop{in} \  \Omega,\\
  & \xi_1\frac{u^{k+1} - u^k}{\tau} + \alpha \partial_n u^{k+1}+\gamma'(u^{k+1})= 0 \quad \text{on}\ \partial \Omega.
\end{eqnarray*}
 In a simple setting with $\gamma(u)=const, \xi_1=0$,  and subject to periodic boundary conditions, we 
are let to
\[ u^{k + 1} = \tmop{argmin}_{u \in H^1 (\Omega)}  \int_{\Omega} \frac{(u -
   u^k)^2}{2 \tau} + \frac{\alpha}{2} | \nabla u |^2 + F(u) \tmop{d\tmmathbf{x}}. \]
If $F (u) = \frac{(1 - u^2)^2}{4}$, we have $f (u) = F' (u) = u^3 - u$.

{\tmstrong{Example 2}}: The Cahn-Hilliard equation.

Consider the Cahn-Hilliard equation described in Example 2 of  Section \ref{AE}, we introduce a space
\[ \tmmathbf{V} \assign \{ \tmmathbf{v} : \Omega \rightarrow \mathbb{R}^m :
   \tmmathbf{v} \in H (\tmop{div}, \Omega), \tmmathbf{v} \cdot \tmmathbf{n} =
   0 \ \tmop{on} \  \partial \Omega \} . \]
According to the Onsager principle, we obtain our time-discrete variational formulation of the Cahn-Hilliard equation: $(u^{k + 1}, \tmmathbf{j}^{k + 1})$ is obtained by solving the following constraint optimization problem: 
\begin{align*}
  \tmop{min}_{u \in H^1 (\Omega),
  \tmmathbf{j} \in \tmmathbf{V}} &   \left\{ \frac{\tau}{2} \int_{\Omega} 
  \tmmathbf{j}^\top M^{-1}(u^k)\tmmathbf{j}  \tmop{d\tmmathbf{x}} + 
  \frac{\xi}{2\tau} \int_{\partial \Omega}
  |u-u^k|^2 dS +
  \int_{\Omega} \frac{\alpha}{2} | \nabla u |^2 +
  F (u) \tmop{d\tmmathbf{x}} +\int_{\partial \Omega}\gamma(u)dS\right\}, \\
  s.t. \quad &  \frac{u - u^k}{\tau} + \nabla \cdot \tmmathbf{j} = 0 \; \text{in}\; \Omega. 
\end{align*}
Here $\tau$ is the time step, with initial data $u^0$. When $\tau$ is small and $u$ has good regularity, the minimizer of the problem may be unique and satisfies the Euler-Lagrange equation.
\begin{eqnarray*}
  \frac{u-u^k}{\tau} & = & \nabla \cdot (M (u^k) \nabla \mu), \quad \
  \tmop{in} \ \Omega,\\
  \mu & = & \alpha \Delta u - F'(u), \qquad \tmop{in} \ \Omega,\\
 M (u^k) \nabla \mu \cdot\tmmathbf{n} & = & 0, \qquad \hspace{5em} \ \tmop{on} \ \partial \Omega,\\
  \xi \frac{u-u^k}{\tau} & = & - (\partial_n u + \gamma'  (u)), \quad \, \tmop{on} \ \partial
  \Omega .
\end{eqnarray*}
This is an implicit time discretizaiton of the Cahn-Hilliard equaiton. 

{\tmstrong{Example 3}}: The Fokker-Planck equation.

For the Fokker-Planck equation, the scheme \eqref{e:discOnsCsv2} is reduced to
\begin{align*}
  (u^{k + 1}, \tmmathbf{m}^{k + 1}) = \tmop{argmin}_{u \in L^1_+ (\Omega),
  \tmmathbf{m} \in \tmmathbf{V}} &  \left\{ \frac{1}{2  \tau} \int_{\Omega} \frac{| \tmmathbf{m}
  |^2}{u^k} \tmop{d\tmmathbf{x}} +\int_{\Omega} \beta^{-1} u \log u + u U (\tmmathbf{x})
  \tmop{d\tmmathbf{x}} \right\}, \\
  s.t. & \ u - u^k + \nabla \cdot \tmmathbf{m} = 0.
\end{align*}
Under appropriate  conditions on $U$, one can verify that the underlying functional is convex with linear constraints so that there exists a unique
minimizer. We could also prove that $u^{k + 1} > 0$, i.e. the scheme is positive preserving.
This discrete formulation serves as a natural approximation to the celebrated {\color{black}Jordan–Kinderlehrer–Otto (JKO)} scheme: 
\begin{align*}
& \text{Determine} \;  u^{k+1} \; \text{that minimizes} \; \left\{ \frac{1}{2\tau}W_2^2(u, u^k)+  \int_{\Omega} \beta^{-1} u \log u + u U (\tmmathbf{x})
  \tmop{d\tmmathbf{x}} \right\},
\end{align*}
where $W_2$ is the 2-Wasserstein distance \cite{JKO1998}.  That is 
$$
W^2(a,b)\sim \inf_{m\in V} \left\{ \int_{\Omega} \frac{|m|^2}{a}\tmop{d\tmmathbf{x}}, \quad b-a+\nabla\cdot m=0 \right\}. 
$$
In fact, such as an approximation is even more clear when comparing with 
\[
W^2(a, b)=\inf_{m, \rho} \left\{ \int_0^1 \int_{\Omega} \rho |v|^2\tmop{d\tmmathbf{x}}\tmop{ds}, \quad \partial_s \rho+\nabla\cdot (\rho v)=0, \; \rho(0)=a, \quad \rho(1)=b \right\}. 
\]
This, called the Benamou–Brenier formula \cite{BB00}, establishes a tight connection between absolutely continuous curves in the probability density space with Wasserstein metric and solutions to the continuity equation. \\

{\tmstrong{Example 4}}: The Plank-Nernst-Poisson equations.

For the Plank-Nernst-Poisson, we have
\begin{align*}
  (\tmmathbf{u}^{k + 1}, \tmmathbf{m}^{k + 1}) = \tmop{argmin}_{\tmmathbf{u}
  \in (L^2 (\Omega))^s, \tmmathbf{m} \in \tmmathbf{V}^s} & \left\{ 
\mathcal{E} (\tmmathbf{u})
+\frac{1}{2\tau}\int_{\Omega}( D_i(\tmmathbf{x})u_i^k)^{-1}|\tmmathbf{m}_i|^2 \tmop{d\tmmathbf{x}}
  \right\} \\ 
  s.t. &\   u_i - u_i^k + \nabla \cdot \tmmathbf{m}_i = 0, \quad i = 1,
  \ldots, s.\\
  &  - \nabla \cdot (\varepsilon_0 \nabla \varphi) = f(\tmmathbf{x})+ \sum_{i = 1}^s z_i u_i,
\end{align*}
where 
\[ \mathcal{E} (\tmmathbf{u}) = \int_{\Omega} \sum^s_{i = 1} u_i \log u_i +
   \frac{\varepsilon_0}{2} | \nabla \varphi (\tmmathbf{x}) |^2 \tmop{d\tmmathbf{x}}. \]
This optimization problem, with two linear constraints, has been derived in \cite{LM23} in several steps by approximating a dynamical formulation of the JKO type scheme \cite{JKO1998}. Refer to \cite{LM23} for further details on an alternative derivation, along with other formulations with different boundary conditions for $\varphi$.   

{\tmstrong{Example 5}}: The Maxwell-Stefan problem.

For the Maxwell-Stefan problem, we are led to
\begin{align*}
  (\tmmathbf{u}^{k + 1}, \tmmathbf{m}^{k + 1}) = \tmop{argmin}_{\tmmathbf{u}
  \in (L^2 (\Omega))^s, \tmmathbf{m} \in \tmmathbf{V}^s} &  
\left\{  \int_{\Omega}
  \sum_{i = 1}^s u_i \log u_i \tmop{d\tmmathbf{x}} + \frac{1}{4 \tau} \int_{\Omega}
  \sum_{i, j = 1}^s b_{i j} u_i^k u_j^k \left| \frac{\tmmathbf{m}_i}{u_i^k} -
  \frac{\tmmathbf{m}_j}{u_j^k} \right|^2 \tmop{d\tmmathbf{x}} \right\},  \\
  s.t. & \ u_i - u_i^k + \nabla \cdot \tmmathbf{m}_i = 0, \quad i = 1,
  \ldots, s; \quad \text{and}\; \sum_{i=1}^s \nabla \cdot \tmmathbf{m}_i=0.  
\end{align*}
Such constraint minimization differs yet similar to  that discussed in \cite{HLTW21}, in which the authors formulate an optimization problem for interpreting the implicit-explicit scheme to the  Maxwell-Stefan problem. 

Again our variational scheme is related to the JKO scheme \cite{JKO1998}, an analogy due to the connection between frictional dissipation and the Wasserstein distance offered by the Benamou--Brenier interpretation \cite{BB00} of the Monge-Kantorovich mass transfer problem. There is however one important difference, as the frictional dissipation is more elaborate in the multi-component mixture situation. The minimizers of the above constraint problem can be calculated by considering the min-max augmented Lagrangian, which upon taking $\tau \tmmathbf{v}_i= \frac{\tmmathbf{m}_i}{u_i^k}$
gives 
\begin{align*}
\min_{u, \tmmathbf{v}}\max_{\alpha, \beta} L(u,\tmmathbf{v},p,\mu)
& =  \int_{\Omega} \sum_{i=1}^s u_i \log u_i + 
  \frac{\tau}{4} \sum_{i, j = 1}^s b_{i j} u_i^k u_j^k \left| \tmmathbf{v}_i -
  \tmmathbf{v}_j \right|^2  \tmop{d\tmmathbf{x}} \\
 & \qquad  - \int_{\Omega}  \tau p  \sum_{i=1}^s \nabla\cdot (\tmmathbf{v}_i u_i^k ) \tmop{d\tmmathbf{x}} 
  - \int_{\Omega}  \sum_{i=1}^s (\mu_i (u_i-u_i^k)
  - \tau \nabla\mu_i \cdot   \tmmathbf{v}_i u_i^k \tmop{d\tmmathbf{x}}. 
\end{align*} 
Computing the variational derivatives, which vanish at the saddle points, we obtain
\begin{align*}
& 1+\log u_i  -\mu_i =0\\
& \sum_{j = 1}^s b_{i j} u_j^k  (\tmmathbf{v}_i - \tmmathbf{v}_j) +\nabla \mu_i  = -\nabla p.    
\end{align*}
We substitute the first relation into the second relation. This leads to 
an explicit-implicit time-discretization: 
\begin{eqnarray*}
   \frac{u_i-u_i^k}{\tau} + \nabla \cdot (u_i^k  \tmmathbf{v}_i) & = & 0, \qquad
  \tmop{in} \ \Omega,\\
  \sum_{j = 1}^s b_{i j} u_j^k  (\tmmathbf{v}_i - \tmmathbf{v}_j) + \nabla \log
  u_i & = & -\nabla p,
  \quad \tmop{in} \ \Omega,\\
\sum_{i=1}^s \nabla \cdot( u_i^k \tmmathbf{v}_i) & =& 0, \quad \tmop{in} \  \Omega.
\end{eqnarray*}
By using the relation~\eqref{e:tempP}, this is equivalent to the explicit-implicit scheme introduced in \cite{HLTW21}: 
\begin{eqnarray*}
  \frac{u_i-u_i^k}{\tau} + \nabla \cdot (u_i^k   \tmmathbf{v}_i) & = & 0, \qquad
  \tmop{in} \  \Omega,\\
  \sum_{j = 1}^s b_{i j} u_j^k  (\tmmathbf{v}_i - \tmmathbf{v}_j) + \nabla \log
  u_i & = & \frac{1}{\sum_{j = 1}^s u_j^k} \sum_{j = 1}^s u_j^k \nabla \log u_j,
  \quad \tmop{in} \  \Omega,\\
  \sum_{i=1}^s  \nabla \cdot( u_i^k \tmmathbf{v}_i)  & = & 0, \quad \tmop{in} \  \Omega.
\end{eqnarray*}

{\tmstrong{Example 6}}: The incompressible and immiscible multi-phase flow in porous media.

For the multiphase flow in porous media, we can derive a discrete scheme
\begin{align*}
  (\tmmathbf{u}^{k + 1}, \tmmathbf{m}^{k + 1}) = \tmop{argmin}_{\tmmathbf{u}
  \in (L^2 (\Omega))^s, \tmmathbf{m} \in \tmmathbf{V}^s} &   \left\{ \int_{\Omega}
  \phi F (\tmmathbf{u}) d x + \frac{1}{\tau} \int_{\Omega} \sum_{i = 1}^s
  \frac{1}{2} \tmmathbf{m}_i^\top  K_i^{- 1} \tmmathbf{m}_i \tmop{d\tmmathbf{x}} \right\} \\
  s.t. &\   \phi (u_i - u_i^k) + \nabla \cdot \tmmathbf{m}_i = 0, \quad i = 1,
  \ldots, s.\\
  &   \sum_{i = 1}^s u_i = 1,
\end{align*}
where $\tmmathbf{m} = \tau \tmmathbf{v}$ and $K_i=K_i(\tmmathbf{u}^k)$. Here $\tau$ is the time step size, with initial data $\tmmathbf{u}^0$ satisfying $\sum^s_{i=1}{u}^0_i=1$. When $\tau$ is small, the minimizer of the above problem may be unique and satisfies the Euler-Lagrange equation which is indeed a first order explicit-implicit time discretizaiton of the system (\ref{e:Porous_a})-(\ref{e:bndPorous}) with the explicit value only for $K_i=K_i(\tmmathbf{u}^k)$.  By leveraging the definition of $F(\tmmathbf{u})$, the constraint $\sum^s_i u_i=1$, and optimization algorithms applied to the aforementioned problem, we observe that the approximation of $u_i$ inherently preserves bounds. Furthermore, the natural mass conservation for each phase is also ensured.


{\tmstrong{Example 7}}: The incompressible and immiscible multi-phase flow in porous media with rock compressibility.  

For the multi-phase flow in porous media with rock compressibility, we can derive a discrete scheme as follows:
\begin{align*}
  (\phi^{k+1},\tmmathbf{u}^{k + 1}, \tmmathbf{m}^{k + 1}) = \tmop{argmin}_{\phi\in L^2(\Omega),\tmmathbf{u}
  \in (L^2 (\Omega))^s, \tmmathbf{m} \in \tmmathbf{V}^s} &   \left\{ \int_{\Omega}
  \phi F (\tmmathbf{u}) \tmop{d\tmmathbf{x}} +\int_\Omega p_e (\phi -\phi^k)\tmop{d\tmmathbf{x}} + \frac{1}{\tau} \int_{\Omega} \sum_{i = 1}^s
  \frac{1}{2} \tmmathbf{m}_i^\top  K_i^{- 1} \tmmathbf{m}_i  \tmop{d\tmmathbf{x}} \right\} \\
  s.t. &\   \phi u_i - \phi^ku_i^k+ \nabla \cdot \tmmathbf{m}_i = 0, \quad i = 1,
  \ldots, s,\\
  &   \sum_{i = 1}^s u_i (\phi-\phi^k) +\sum_{i = 1}^s \nabla \cdot \tmmathbf{m}_i = 0,
\end{align*}
where $\tmmathbf{m} = \tau \tmmathbf{v}$ and $K_i=\lambda_i(\tmmathbf{u}^k)K(\phi^k)$. The second constraint can also be written as $\sum^s_{i=1}u_i=1$. The initial data $\tmmathbf{u}^0$ satisfying $\sum^s_{i=1}{u}^0_i=1$. We apply $\phi = \phi_r e^{\gamma(p_e-p_r)}$ and $K =  K_0 \frac{\phi^3(1-\phi_r)^2}{\phi_r^3(1-\phi)^2}$ in the above optimization problem. Thus when $\tau$ is small, the minimizer of the above problem may be unique and satisfies the Euler-Lagrange equation which is a first order explicit-implicit time discretizaiton of the nonlinear system (\ref{pm_rock_1})-(\ref{pm_rock_7}) with the explicit value only for $K_i=\lambda_i(\tmmathbf{u}^k)K(\phi^k)$. 
 The preservation of bounds for $u_i$ and mass conservation for each phase are also maintained through the approximation of the aforementioned minimization problem.


\section{Numerical solutions}
The core idea of our numerical schemes is to discretize each semi-discrete variational formulation in space, and then apply an efficient algorithm to optimize the problem. It suffices to describe the discretizaiton of the first step, from the datum $\tmmathbf{u}^0$ to the minimizer $\tmmathbf{u}^1$.

\subsection{The discretized conserved problem} 
We consider only the general problem with mass conservation, 
\begin{align}
  (\tmmathbf{u}^{1}, \tmmathbf{m}^{1}) = \tmop{argmin}_{\tmmathbf{u}\in (L^2(\Omega))^s,
  \tmmathbf{m} \in \tmmathbf{V}^s} &  \left\{   \frac{1}{\tau} \Phi^0_{\tau}
  (\tmmathbf{u}, \tmmathbf{m})+\mathcal{E} (\tmmathbf{u}) \right\}, \\
  s.t. & \quad 
  u_i - u_i^0 -\nabla\cdot \tmmathbf{m}_i=0. 
\end{align}
We shall use either finite difference or finite element for spatial discretization depending on the domain setup.     

{\bf Finite difference}:\\
Let the domain be a box $\Omega=[0, 1]^d$, functions on $\Omega$ extend periodically. For simplicity, we explain the discretization in $d=1$ space dimensions, the translation to $d>1$ is straightforward but notationally cumbersome.  We use a cartesian grid with $N$ cells $I_j=[x_{j-1/2}, x_{j+1/2}]$, with uniform grid step $h=1/N$ and cell center $x_j=x_{j-1/2}+ 0.5h$, for $j=1,\ldots, N$. Functions $w$ are discretized by finite sequences $(w_j)_{j=1, \ldots, N}$ with $w_j\sim w(x_j)$. We define the difference operator by 
$$
(D_h w)_{j+1/2}=\frac{w_{j+1}-w_{j}}{h}, \quad 
(d_h w)_j=\frac{w_{j+1/2}-w_{j-1/2}}{h}
$$
and average by 
$$
\hat w_j=\frac{w_{j+1/2}+w_{j-1/2}}{2}. 
$$
With these notations, the variational scheme for FPE is now discretized as follows: 
\begin{align*} 
 & \min_{u, m} L_h(u, m):=  \left\{ \frac{h}{2 \tau} \sum_{j=1}^N \frac{|\hat m_j|^2}{u_j}  + 
 h \sum_{j=1}^N \beta^{-1} u_j \log u_j + u_j U (x_j)\right\}, \\
  s.t. & \qquad 
  u_j - u^0_j +(d_h m)_j=0. 
\end{align*}
For other application examples, the discrete objective function can be similarly obtained.

{\bf Finite element}: \\
We can also employ the finite element method for spatial discretization. 
Let $\mathcal{T}_h$  represent a regular triangulation of $\Omega$ with mesh size $h$. Let $U_h$ and $V_h$ are proper finite element spaces for $u_{h,i}$ and $\tmmathbf{m}_i$, respectively. Notice that $V_h$ is a finite element space for vector-valued functions.
Then the fully discrete problem is defined as follows,
\begin{align}
  (\tmmathbf{u}_h^{1}, \tmmathbf{m}_h^{1}) = \tmop{argmin}_{\tmmathbf{u}_h\in (U_h)^s,
  \tmmathbf{m} \in (V_h)^s} &  \left\{ 
   \frac{1}{\tau} \Phi^0_{\tau}
  (\tmmathbf{u}_h, \tmmathbf{m}_h) +\mathcal{E} (\tmmathbf{u}_h) \right\}, \\
  s.t. &\ \int_{\Omega} (u_{h,i} - u_{h,i}^0 -\nabla\cdot \tmmathbf{m}_{h,i}) v_h  \tmop{d\tmmathbf{x}} = 0, \quad \forall v_h\in U_h.\nonumber
\end{align}
In applications, we can choose the finite element spaces  so that the fully discrete system is well-posed. 

\subsection{Solution  by optimization  algorithms} 
In this section, we delve into numerical techniques for solving the constraint optimization problem denoted by (\ref{e:discOnsCsv2}).  Let $\theta=(u, m)$, then the problem takes the form 
\begin{align}\label{minp}
\min_{\theta} L_h(\theta), \quad \text{subject to} \quad  B\theta=b, 
\end{align}
where $L_h\in C^1(\mathbb{R}^n)$ is bounded below, 
the constraint set is the linear system corresponding to the discretized PDE constraints. $B\in M^{l\times n}$ is a matrix with $l \leq n$ and $b\in \mathbb{R}^l$ is a vector.  A straightforward method to tackle this constraint optimization is through the following update: 
\begin{equation}\label{eq:pre-gd}
    \theta_{k+1}=\theta_k -\eta G \nabla L_h(\theta_k),
\end{equation}
where the projection matrix $G$ is defined by 
$$
G=I-B^\top (BB^\top)^{-1}B,
$$
ensuring  
$B\theta_{k+1}=b$ if $B\theta_k=b$. 
Here,  $\eta$ is a step size (or learning rate), crucial for  the algorithm's convergence due to the gap between the continuous gradient flow equation and the discrete iteration. Typical choices for $\eta$ are either empirical schedules or damping techniques. 

For enhanced efficiency, we also employ the AEPG algorithm (adaptive energy-based preconditioned gradient descent) introduced in \cite{LNTY23}, which reads:
\begin{subequations}\label{aeng0}
\begin{align}
& v_k= G \nabla l(\theta_k),\\
& r_{k+1} =\frac{r_k}{1+2\eta \|v_k\|^2},\\
& \theta_{k+1}=\theta_k-2\eta r_{k+1}v_k,
\end{align}
\end{subequations}
where $l(\theta)=\sqrt{L_h(\theta)+c}$, $c\in \mathbb{R}$  such that $\inf\limits_{\theta \in \Theta} \left( L(\theta)+c\right)>0$, and $\eta>0$ is the base step size. One striking feature of this algorithm is its unconditional energy stability, i.e.,  $r_k$ as an approximation of $\sqrt{L_h+c}$ is decreasing in $k$ 
for any $\eta>0$.  Consult \cite{LNTY23} for further details of this algorithm.  

{ In certain model scenarios where  computational complexity is not prohibitive,  its feasible to directly tackle the corresponding Euler-Lagrange equation using iterative methods. In such cases, the problem~\eqref{minp} transforms  into a nonlinear equation with a Lagrange multiplier $\lambda$: 
\begin{align*}
    & \nabla L_h(\theta) +B^\top\lambda=0,\\
    &B\theta=b.
\end{align*}
We assume that $B$ satisfies the inf-sup condition, ensuring the well-posedness of  the above problem. A standard Newton scheme for solving the nonlinear equation can be expressed as: 
\begin{align}\label{e:Newton}
   \left(\begin{array}{c}
        \theta^{k+1} \\
        \lambda^{k+1}
   \end{array}\right)=
   \left(\begin{array}{c}
        \theta^{k} \\
        \lambda^{k}
   \end{array}\right)-   
   \left(\begin{array}{cc}
        \nabla^2 L_h(\theta^k) &  B^\top\\
        B & 0
   \end{array}\right)^{-1}
   \left(\begin{array}{c}
         \nabla L_h (\theta^k)+B^\top\lambda^k \\
        B\theta^k-b
   \end{array}\right).
\end{align}
}

\begin{remark}
The optimization algorithms described above offer flexibility in space discretizations. For instance, when employing the mixed finite element method to solve two-phase flow in porous media, as demonstrated in Example 6 and Example 7, conventional approaches require the use of upwind schemes in space discretization for mass conservation equations via semi-implicit schemes \cite{KounS2022,KouWChenS}. However, with the optimization algorithms 
outlined, we find that the imposition of  upwind schemes in the space discretization of mass conservation equations is unnecessary. Its worth noting a recent advancement in \cite{LiuOLiShu2023}, presenting a novel approach for numerically solving time-dependent conservation laws using implicit schemes via primal-dual hybrid gradient methods. In our research, we focus on using the Onsager variational principle as an approximation tool, wherein  the minimization of the discrete Rayleighian functional is achieved through   optimization algorithms.
\end{remark}

\subsection{Simulations} 
The numerical experiments for single equations can be found in literature. We
present only some examples for PNP and the two-phase porous media equations. They
include more than one component.



In the following numerical experiments, we choose the finite element method and implement the schemes in Netgen/NGSolve(\cite{schoberl1997netgen,schoberl2014c}). We choose to use the method~\eqref{e:Newton} to solve 
the optimization problem~\eqref{minp} in the numerical experiments below.

\textbf{Example 1.}
In the first example, we consider the PNP equation in a square region $(0,1)\times(0,1)$. The periodic boundary conditions are proposed for the system. Assume there exist two components in the system and one has positive charges and the other has negative changes. The initial distributions are given respectively by $u_1=1.02+\sin(2\pi x)*\cos(2\pi x)$ and $u_2=1.02+\sin(2\pi y)*\cos(2\pi y)$. We would like to compute the evolution of the two components.

\begin{figure}[htbp]
    \centering
   {\subfigure
   {\resizebox{0.2\columnwidth}{!}{\includegraphics{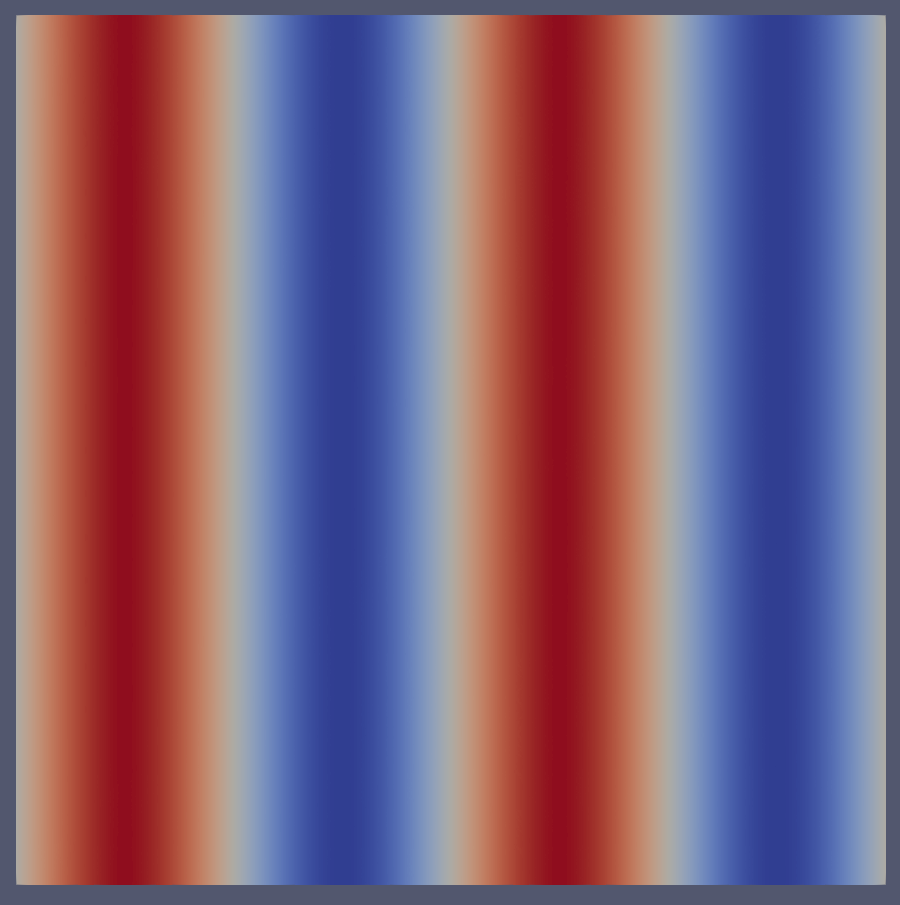}}}}
       {\subfigure
   {\resizebox{0.2\columnwidth}{!}{\includegraphics{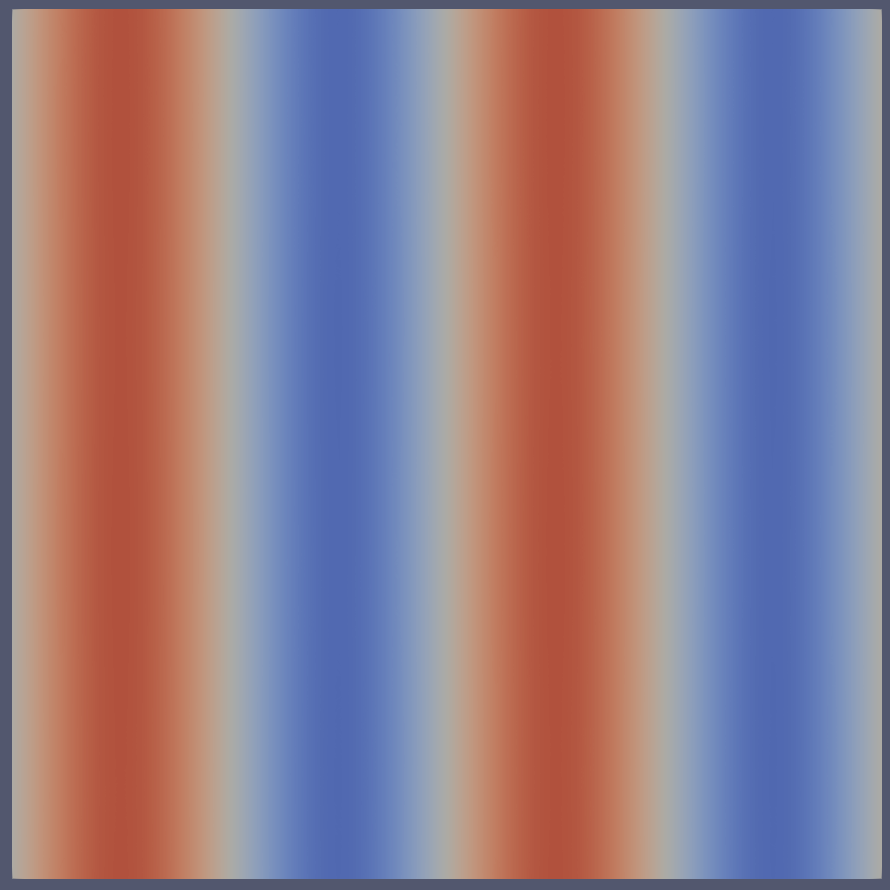}}}}
            {\subfigure
   {\resizebox{0.2\columnwidth}{!}{\includegraphics{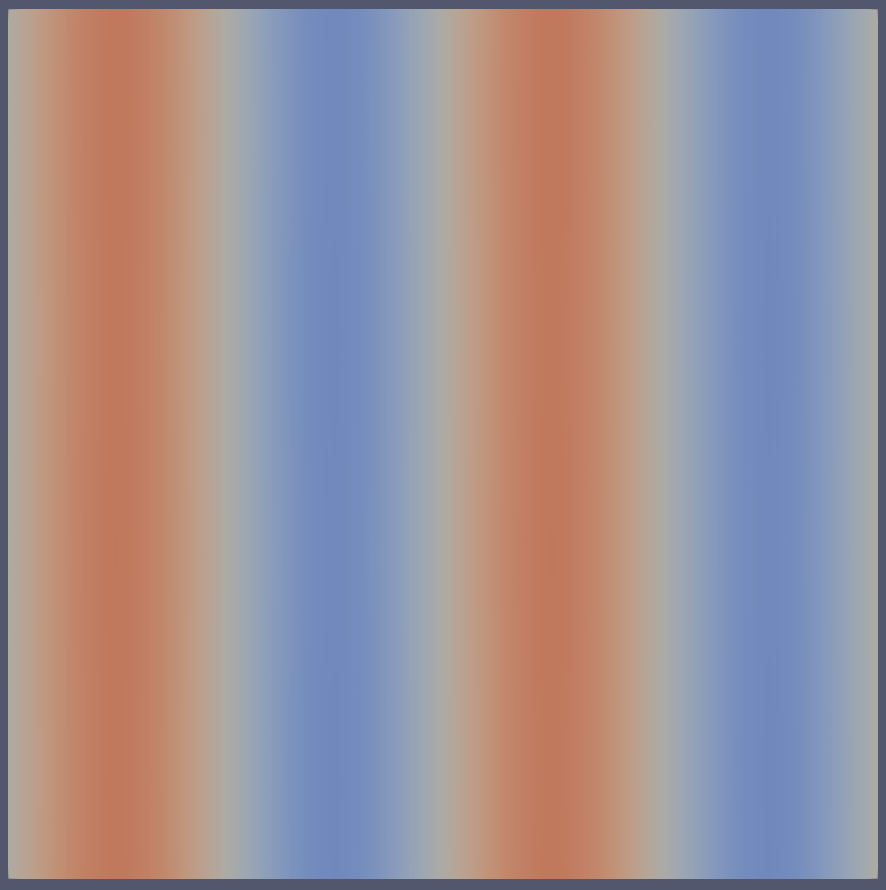}}}}
         {\subfigure
   {\resizebox{0.06\columnwidth}{!}{\includegraphics{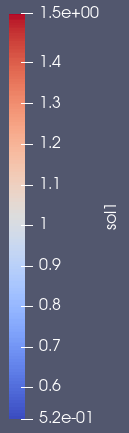}}}}
   
    {\subfigure
    {\resizebox{0.2\columnwidth}{!}{\includegraphics{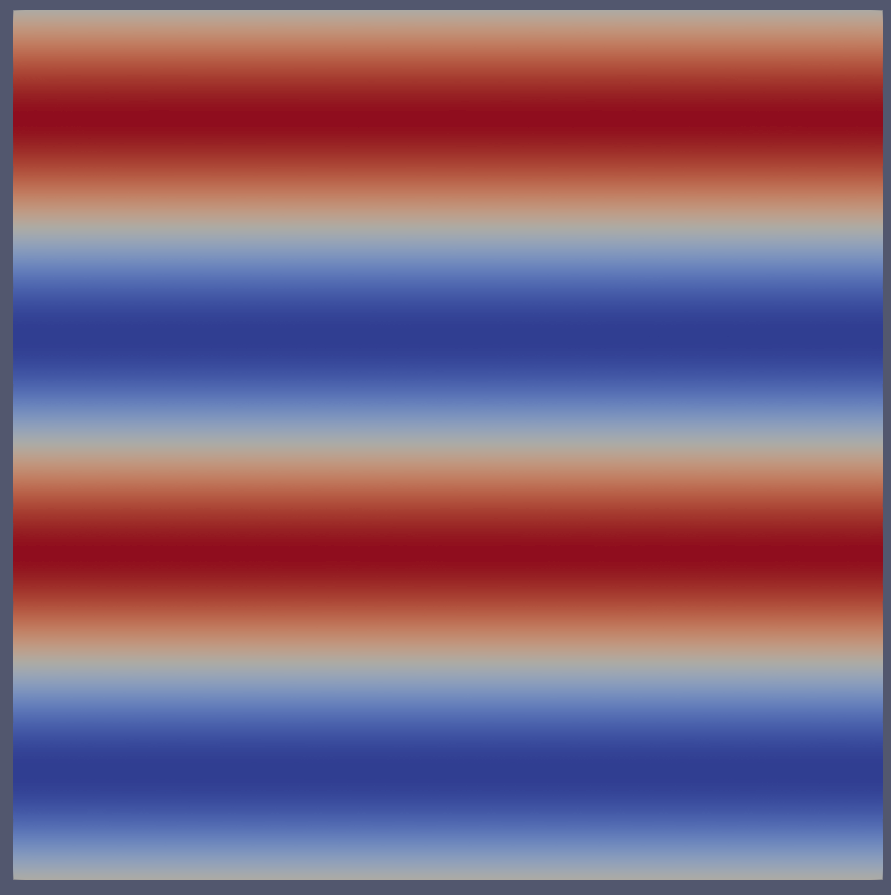}}}}
    {\subfigure
    {\resizebox{0.2\columnwidth}{!}{\includegraphics{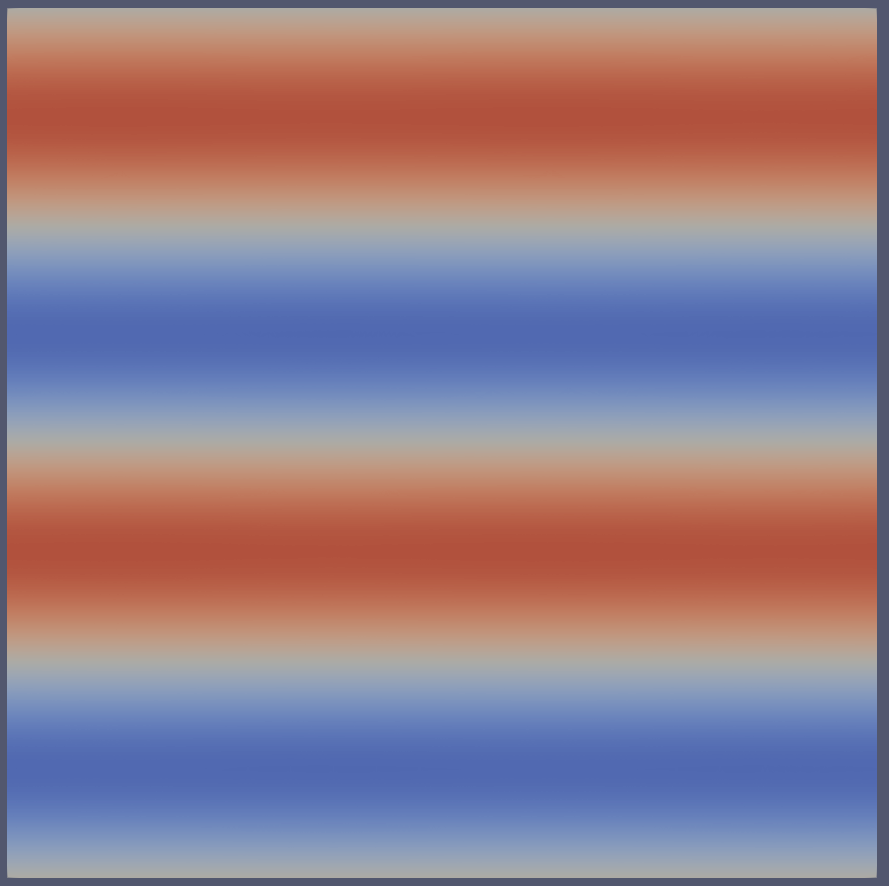}}}}
        {\subfigure
    {\resizebox{0.2\columnwidth}{!}{\includegraphics{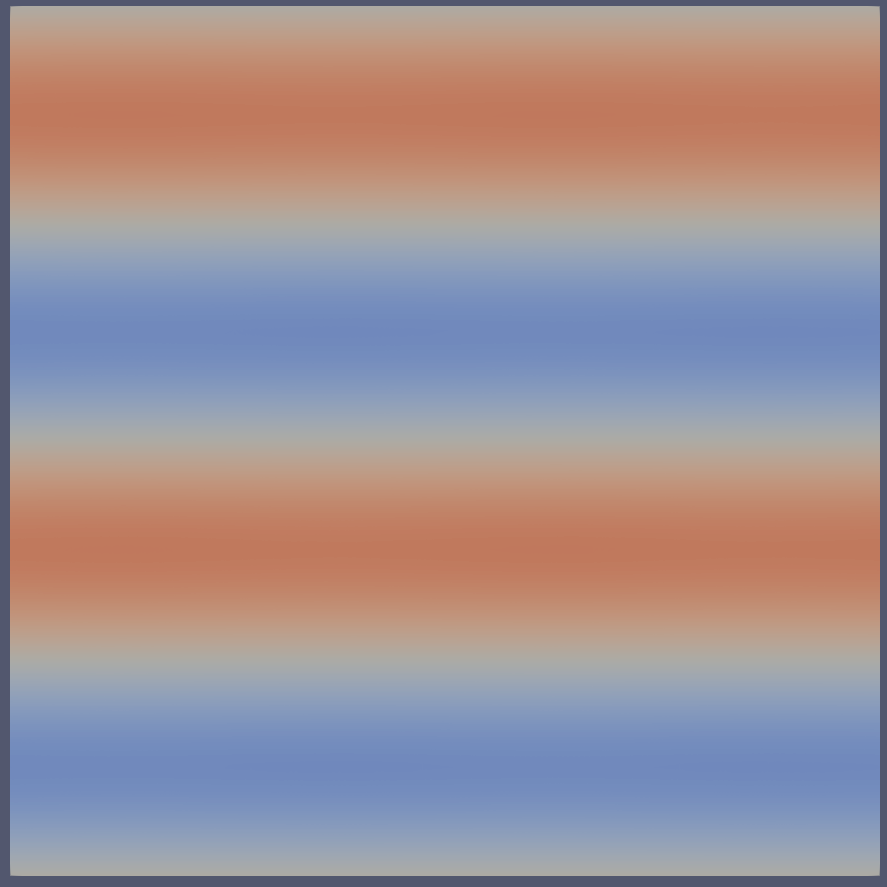}}}}
             {\subfigure
   {\resizebox{0.066\columnwidth}{!}{\includegraphics{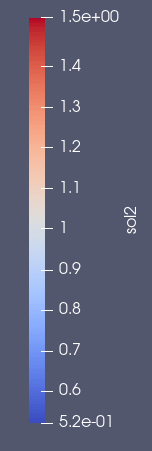}}}}

      {\subfigure
   {\resizebox{0.2\columnwidth}{!}{\includegraphics{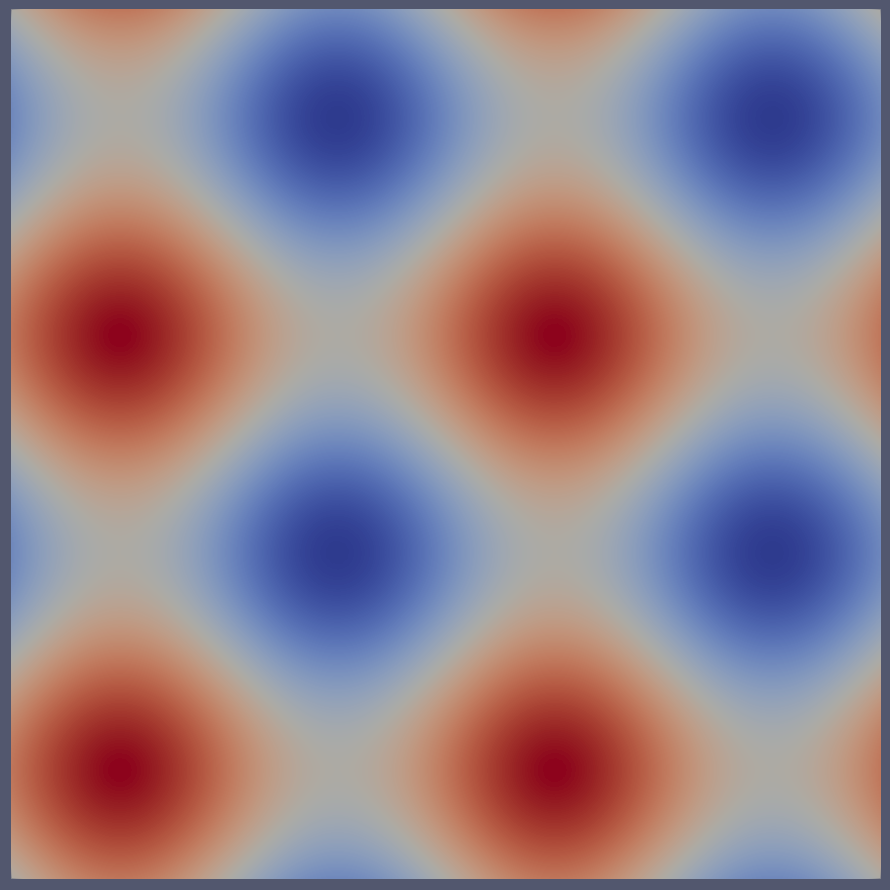}}}}
   {\subfigure
   {\resizebox{0.2\columnwidth}{!}{\includegraphics{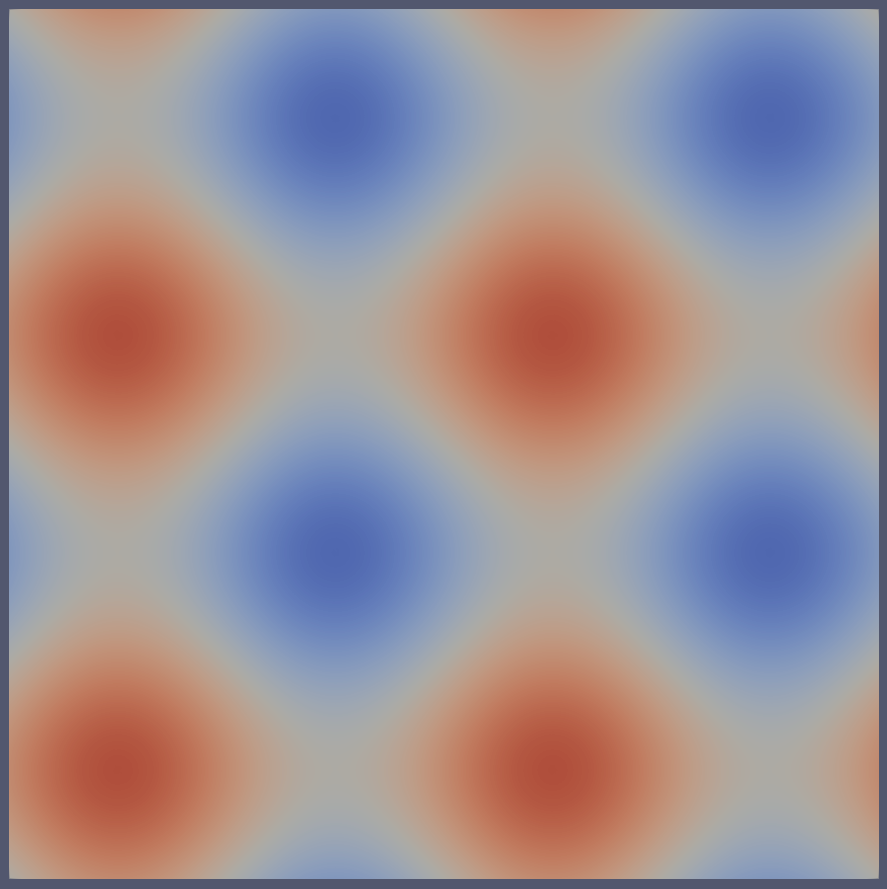}}}}
   {\subfigure
   {\resizebox{0.2\columnwidth}{!}{\includegraphics{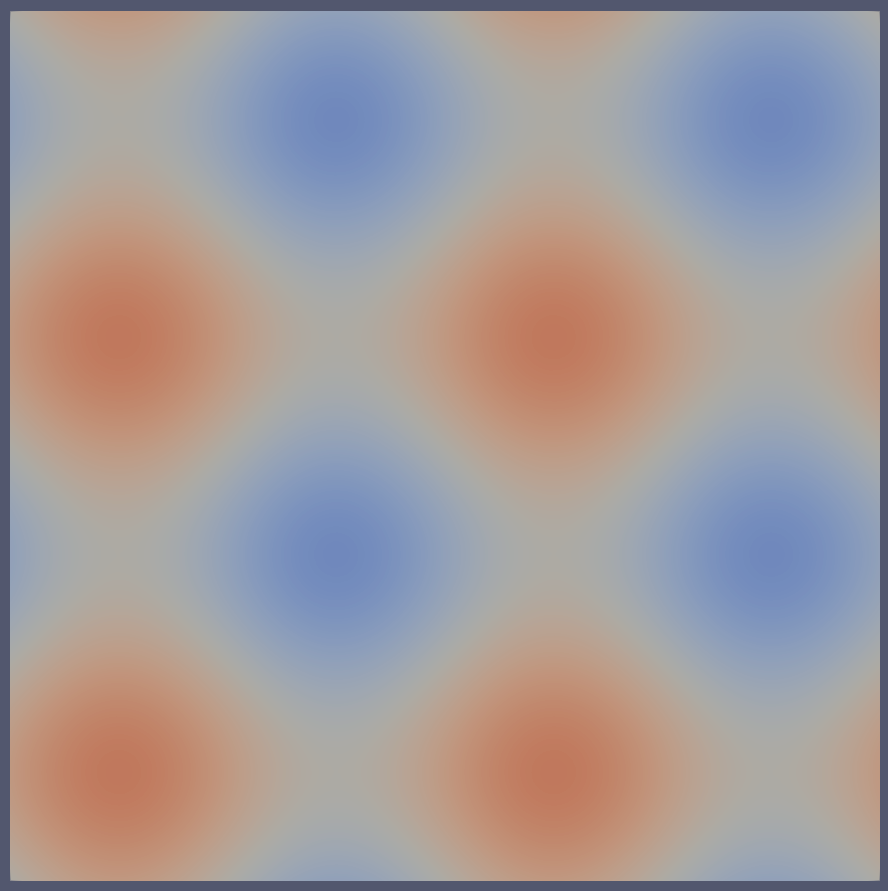}}}}
            {\subfigure
   {\resizebox{0.055\columnwidth}{!}{\includegraphics{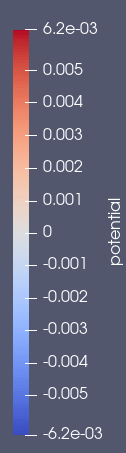}}}}
    \caption{The solutions of the PNP equation. 
    First Row: the distribution of $u_1$ at various time $t=0.0005,0.0025,0.005$; 
    Second Row: the distribution of $u_2$ at various time
    $t=0.0005,0.0025,0.005$;
    Last Row: the potential $\psi$ at various time
    $t=0.0005,0.0025,0.005$.
    }
    \label{fig:ex1}
\end{figure}

\begin{figure}[htbp]
    \centering
     {\subfigure
   {\resizebox{0.6\columnwidth}{!}{\includegraphics{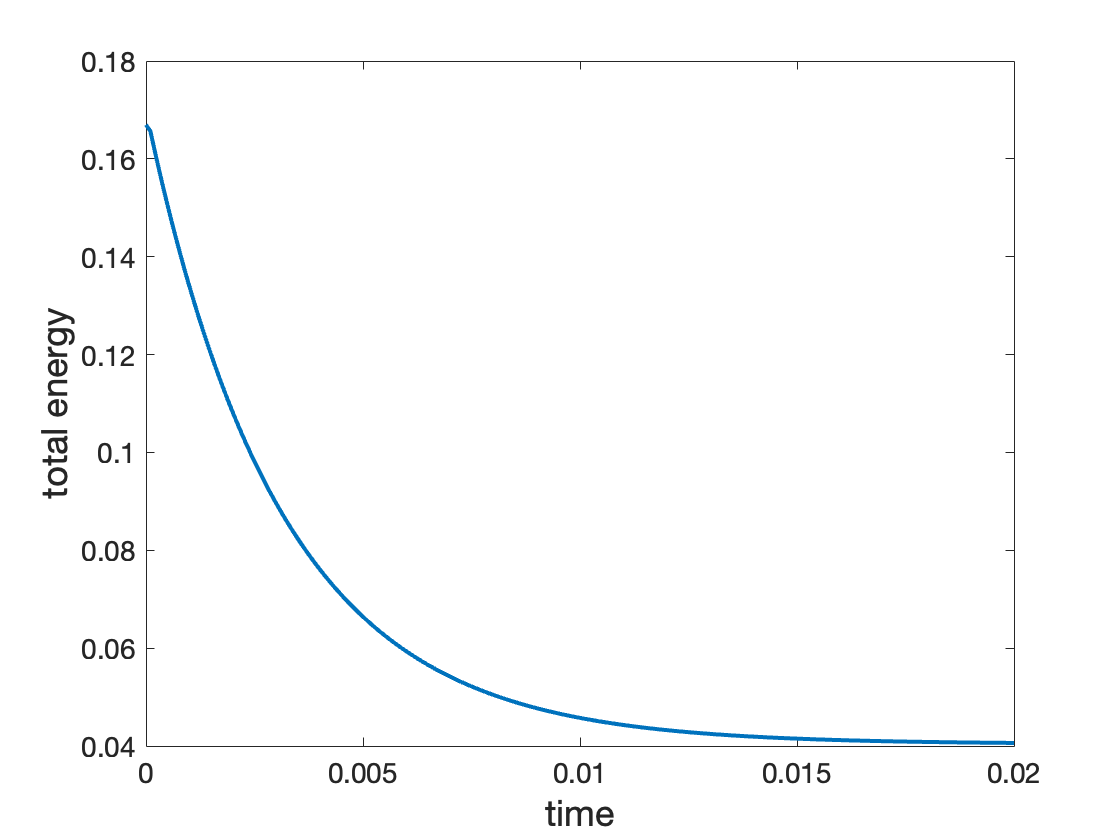}}}}
    \caption{The change of the total energy with respect to time.}
    \label{fig:energyEx1}
\end{figure}

\begin{figure}[htbp]
    \centering
     {\subfigure
   {\resizebox{0.7\columnwidth}{!}{\includegraphics{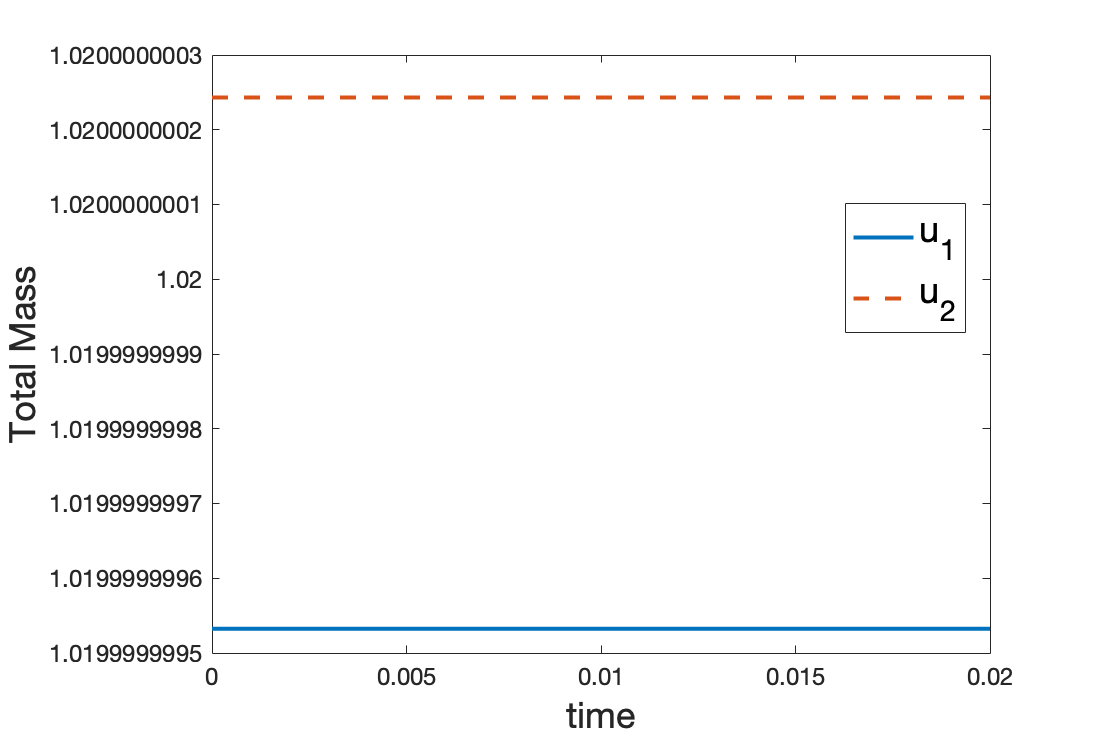}}}}
    \caption{The change of the total mass of the two components with respect to time.}
    \label{fig:massEx1}
\end{figure}

We partition the domain uniformly with mesh size $h=0.05$. We use $P_2$ finite elements to discretize  $u_i$ and $\lambda_i$,
and use $P_3$ finite elements for $\mathbf{m}_i$ and $\phi_i$.
We set $\tau=0.0001$. {In the numerical tests, we typically  choose $\tau\sim c h^2$, where $c$ is a constant. This choices are based on the fact that our scheme is first order with respect to $\tau$ and second order with respect to $h$. However, if our focus is only on stationary states, we can select a larger $\tau$.  It is important, though, that $\tau$ does not become  too large, as this could affect the uniqueness of the minimizer in the optimization problems. Furthermore, the choice of $\tau$ may also influence the convergence of the Newton iterations.}  In each time step, we use a Newton method to solve the nonlinear algebraic equation. It turns out the Newton method converges fast and only two or three iterations are needed in each time step. 
The discrete solution is shown in Figure~\ref{fig:ex1}. We could see that the distributions of the two ions become more and more homogeneous
due to the diffusion.
In Figure~\ref{fig:energyEx1}, we show the change of the total energy with respect to time. We could see that the energy always decays. This verifies the theoretical analysis in Proposition 2 in Section 3. In Figure~\ref{fig:massEx1}, we show the change of the total mass of the two components with respect to time. We could see that the total mass for each component keeps constant.

\textbf{Example 2.} \label{ex-1}
In this example, we consider two-phase flow in porous media within a closed system in a square region $[0, 100 \text{ m}]^2$ with constant porosity. We utilize the data as in \cite{KounS2022}. The initial distribution of wetting-phase saturation and permeability are illustrated in Figure \ref{fig1-initial}. In a porous medium, the porosity in the high-permeability region is 0.3, while the porosity in the rest region is 0.15. The energy parameter in the high-permeability region are given as $\gamma_w =11.655$ bar, $\gamma_n =1.0796$ bar, $\gamma_{wn} = 7.424$ bar, while the energy parameter in the low-permeability region are $\gamma_w = 5.8275$ bar, $\gamma_n = 0.5398$ bar, $\gamma_{wn} = 3.721$ bar. The viscosities are taken as $\eta_w = 0.9$ cP and $\eta_n = 0.1$ cP, respectively. 

The relative permeability is obtained from the following equation:
$$
k_{rw}(S_w) =S_w^3,\quad k_{rn} = (1 - S_w)^3 .
$$
 
\begin{figure}[htbp]
\centering
	\subfigure[Initial wetting-phase saturation]{
		\centering
		\includegraphics[width=1.9in]{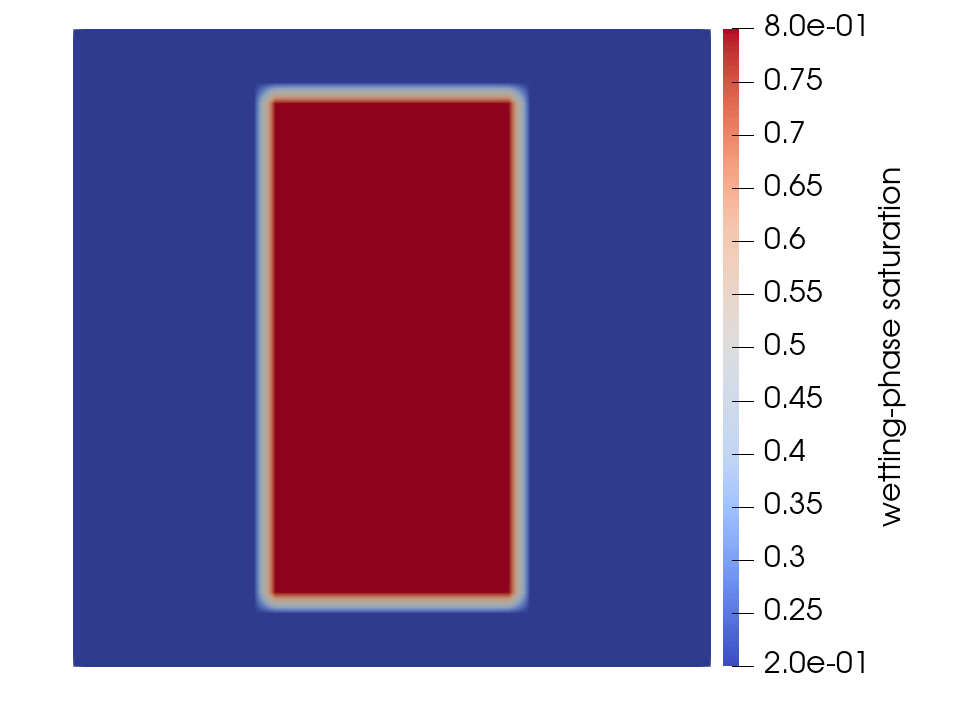}
		\label{fig1}
	}
	\subfigure[Initial permeability]{
		\centering
		\includegraphics[width=1.9in]{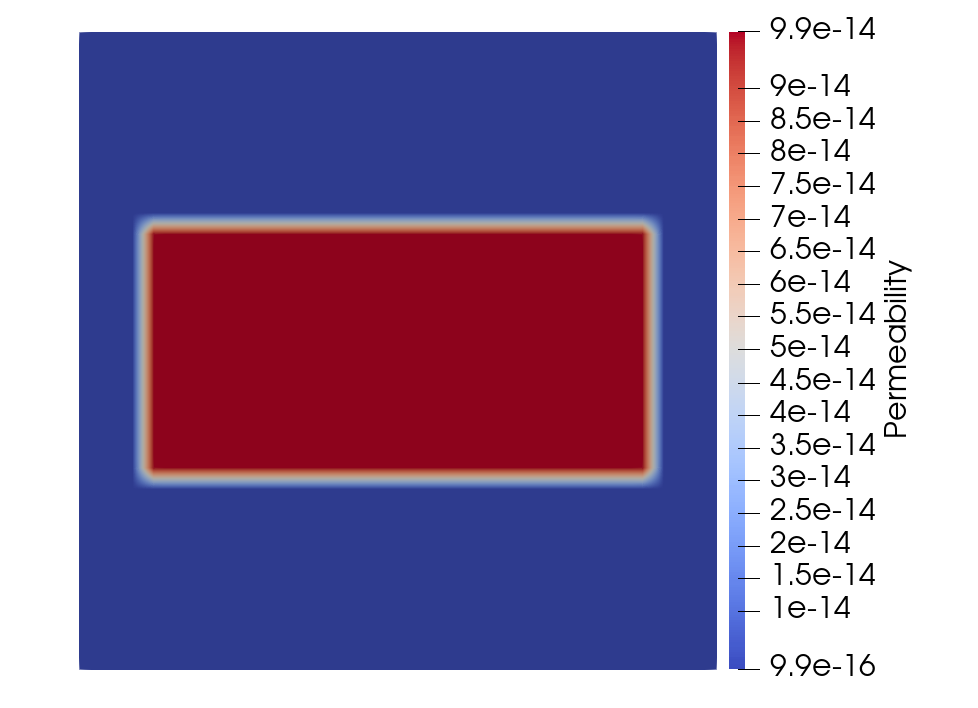}
		\label{fig2}
	}
	\caption{Initial distributions of wetting-phase saturation and permeability in Example 2.} \label{fig1-initial}
\end{figure}

\begin{figure}[htbp]
		\centering
		\includegraphics[width=3.9in]{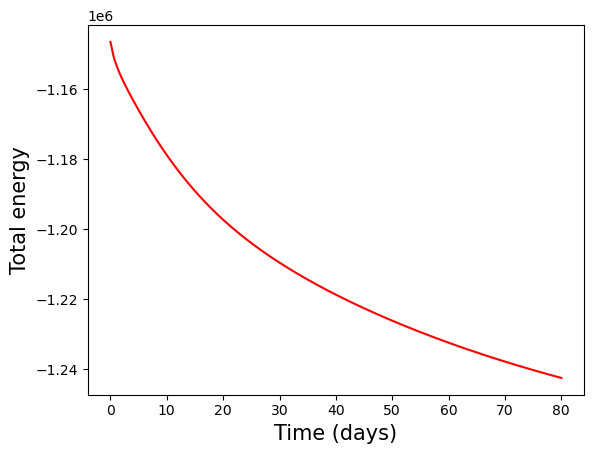}
		\label{fig3}
	\caption{Energy dissipation with time in Example 2.} \label{fig1_energy}
\end{figure}

\begin{figure}[htbp]
	\centering
	\includegraphics[width=5cm]{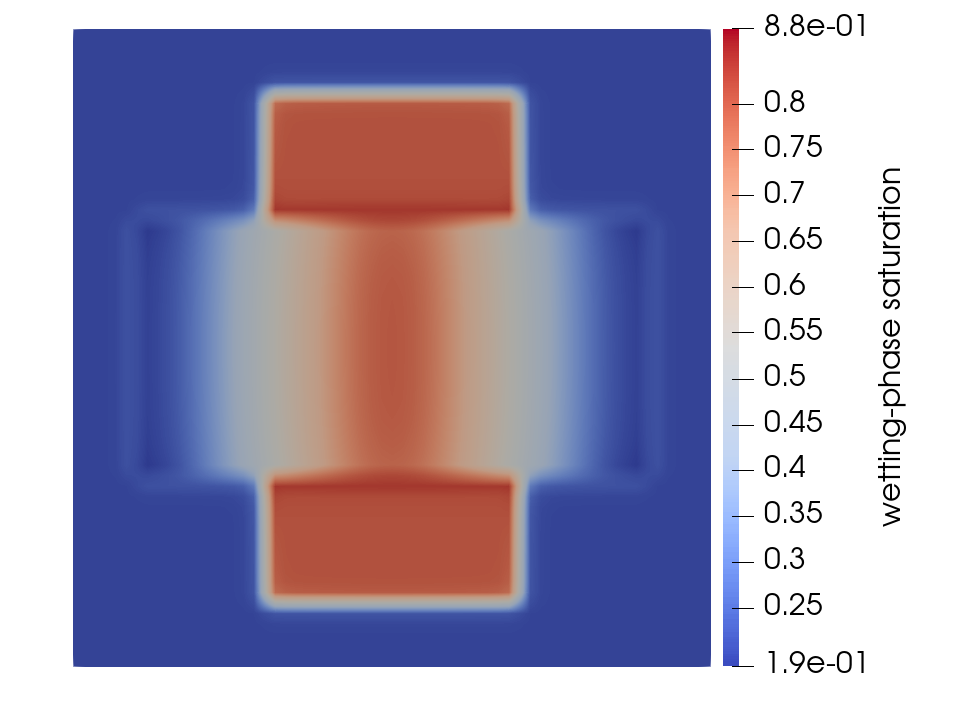}
	\includegraphics[width=5cm]{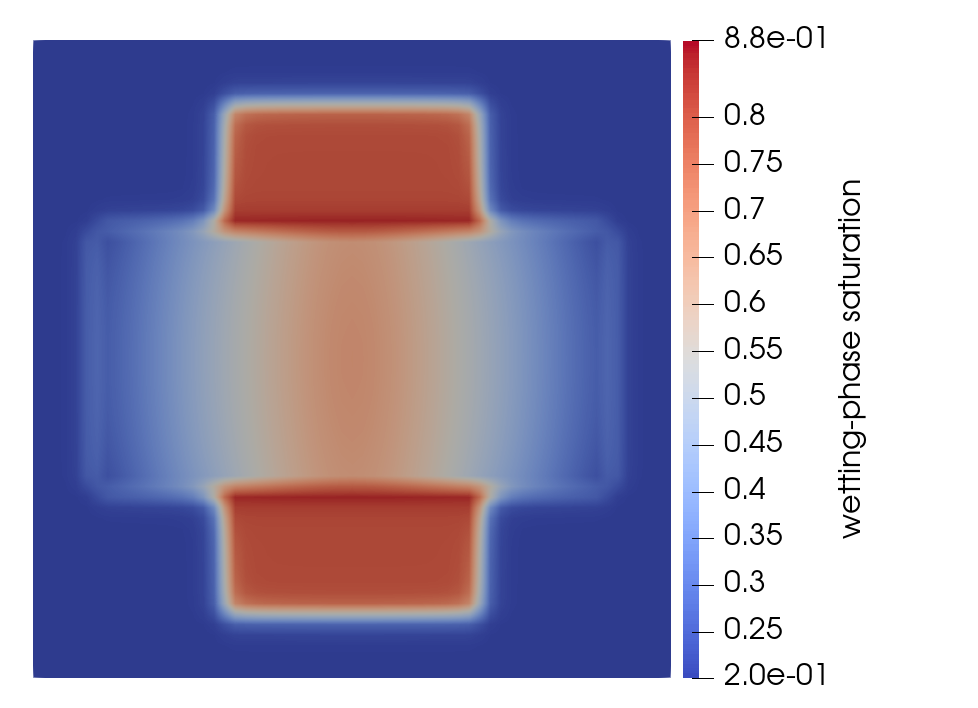}
	
	\includegraphics[width=5cm]{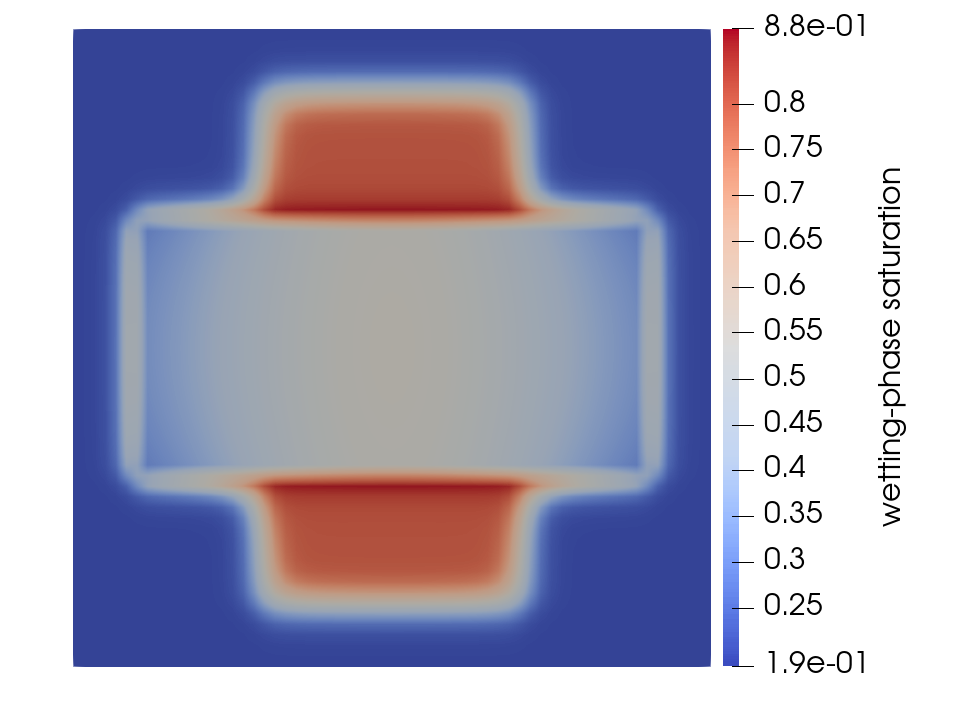}
	\includegraphics[width=5cm]{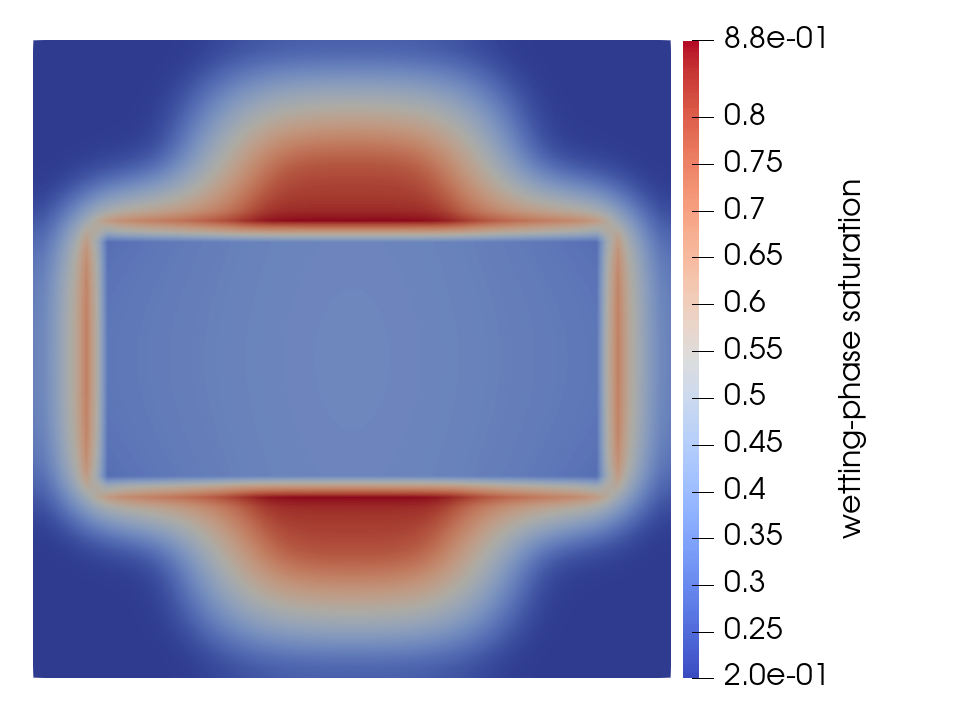}
	\caption{Distributions of wetting-phase saturation at different times in Example 2. Top-left: $t$ = 5 days. Top-right: $t$ = 10 days. Bottom-left: $t$ = 20 days. Bottom-right: $t$ = 80 days.}\label{fig1-Sw}
\end{figure}

\begin{figure}[htbp]
	\centering
	\includegraphics[width=5cm]{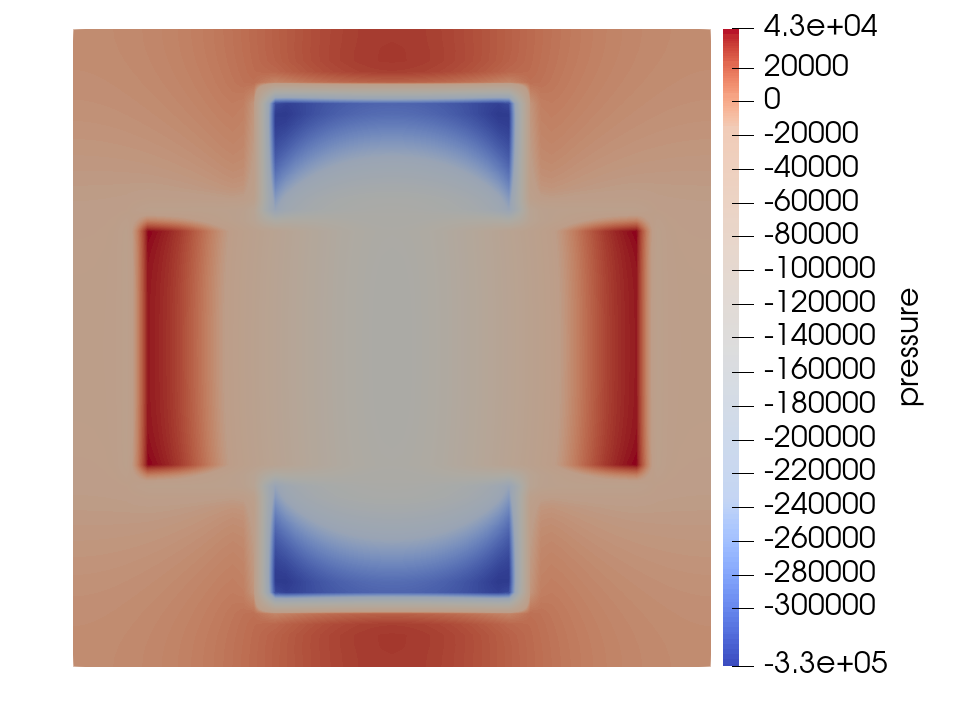}
	\includegraphics[width=5cm]{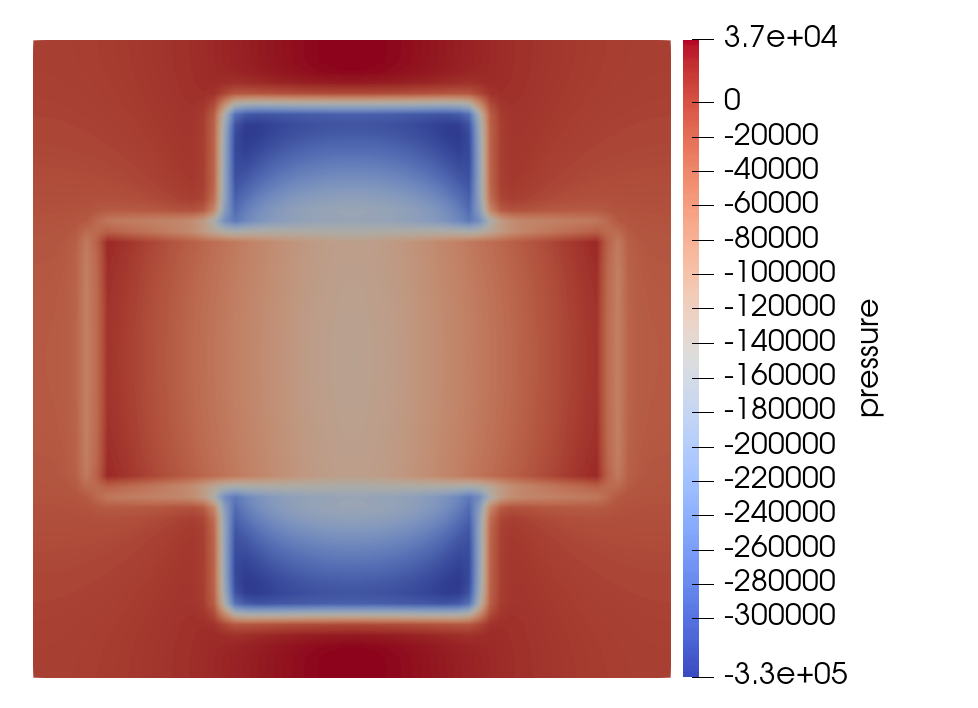}
	
	\includegraphics[width=5cm]{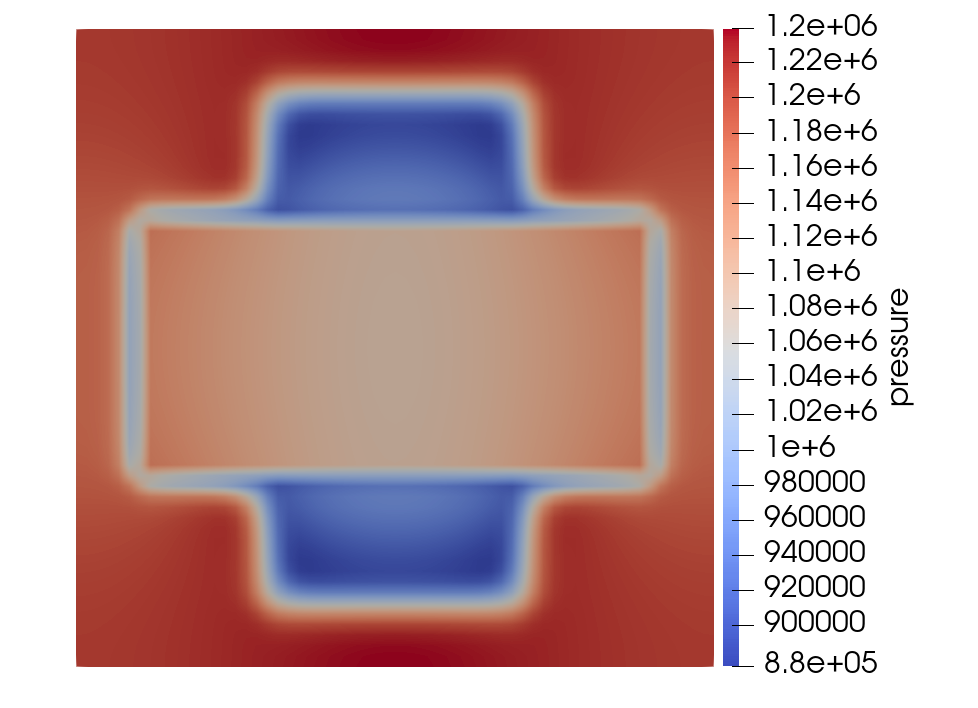}
	\includegraphics[width=5cm]{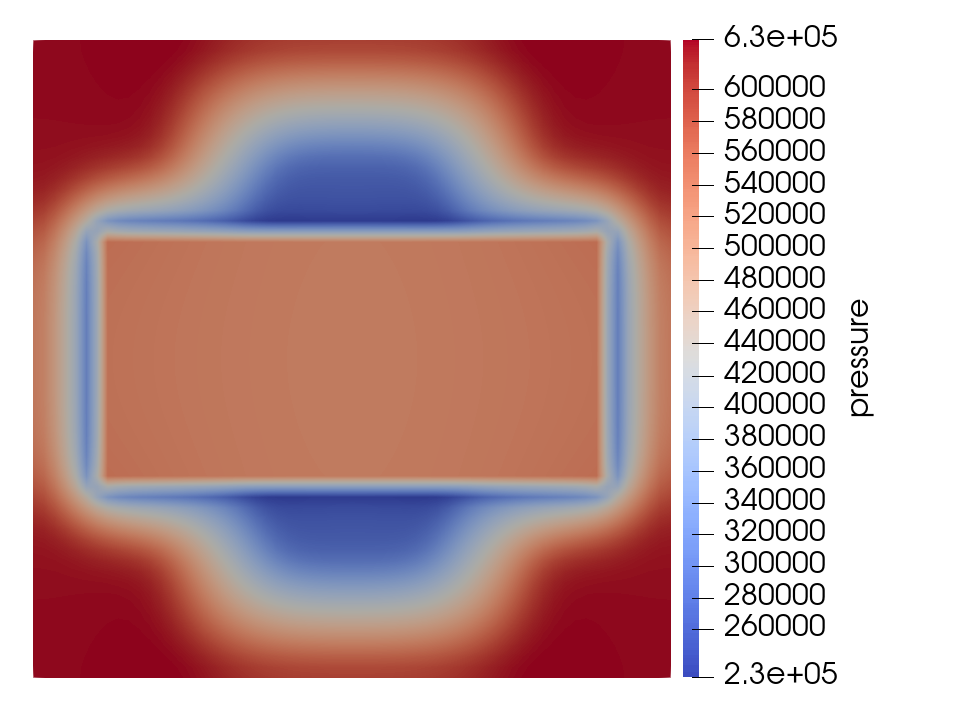}
	\caption{Distributions of pressure at different times in Example 2. Top-left: $t$ = 5 days. Top-right: $t$ = 10 days. Bottom-left: $t$ = 20 days. Bottom-right: $t$ = 80 days.}\label{fig1-p}
\end{figure}

\begin{figure}[htbp]
	\centering
	\includegraphics[width=5cm]{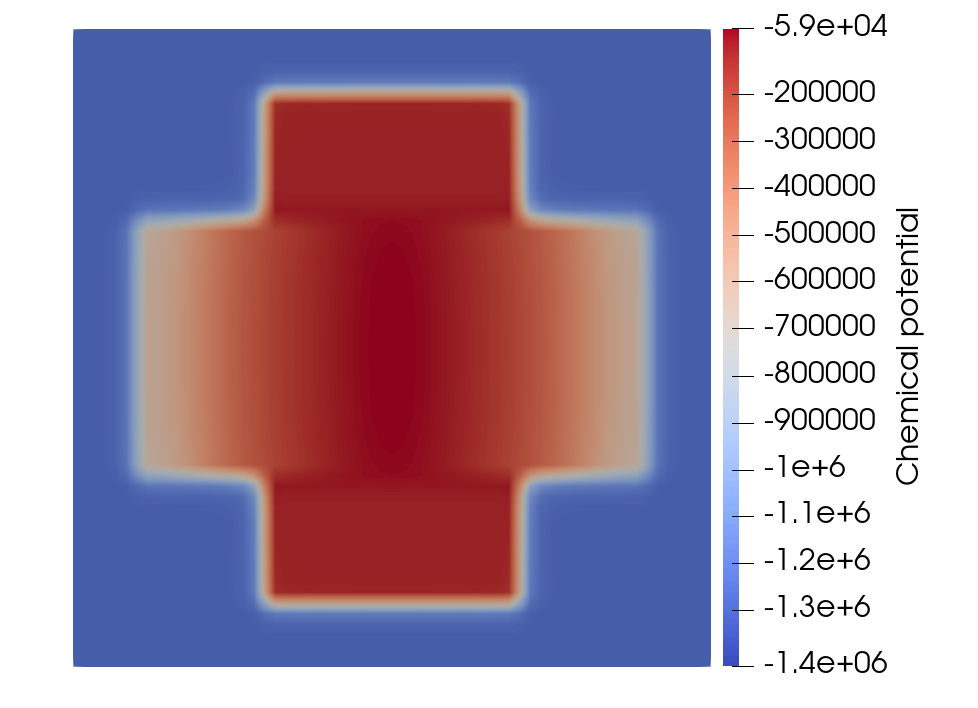}
	\includegraphics[width=5cm]{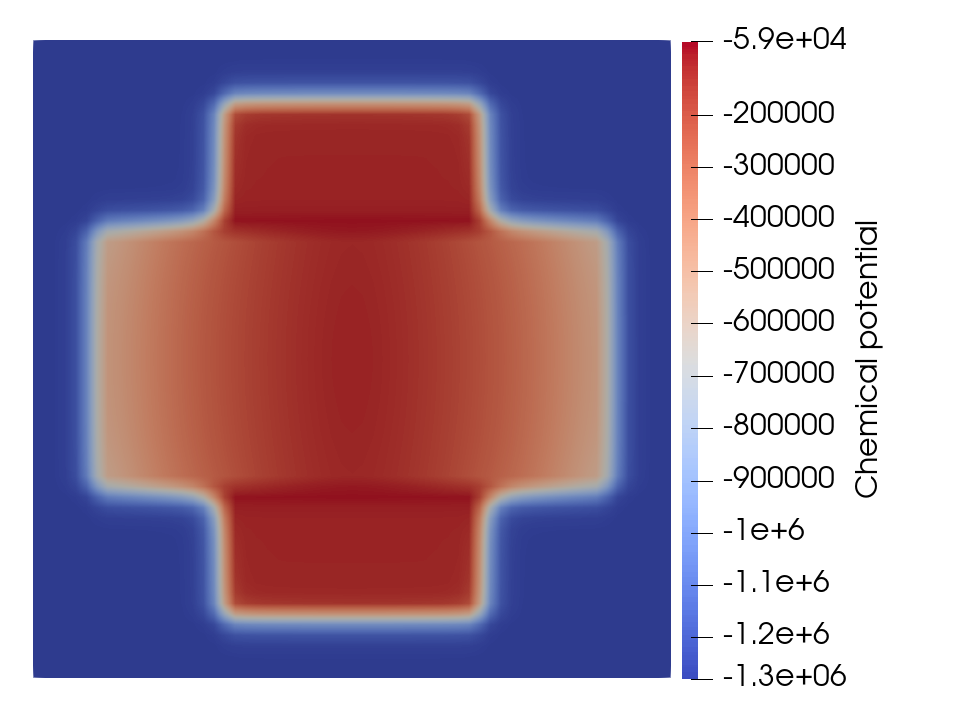}
	
	\includegraphics[width=5cm]{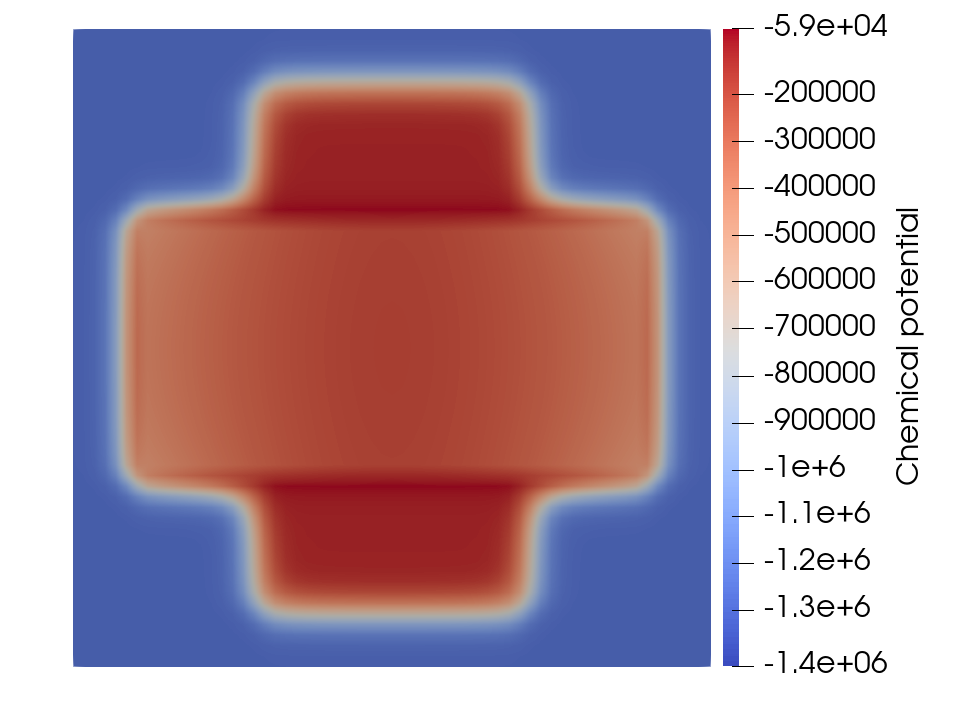}
	\includegraphics[width=5cm]{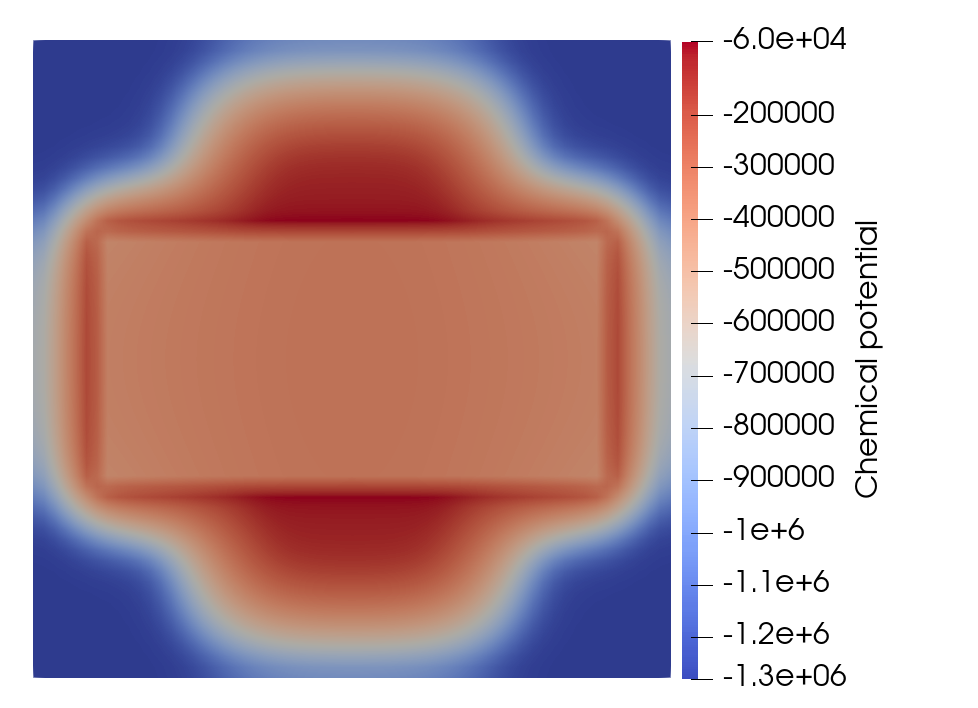}
	\caption{Distributions of chemical potential of wetting-phase at different times in Example 2. Top-left: $t$ = 5 days. Top-right: $t$ = 10 days. Bottom-left: $t$ = 20 days. Bottom-right: $t$ = 80 days.}\label{fig1-chw}
\end{figure}

We use the uniform mesh with $60 \times 60$ grid cells, the time step size is taken as $\tau = 0.01$ day. { The choice of the time step size $\tau$ in the following examples is similar to Example 1 which is related to the mesh size. We choose a suitable time step size to ensure the convergence of optimization algorithm.}  Figure \ref{fig1_energy} shows that the total free energy is  monotonously decreasing with time. The saturation distribution at different times is shown in Figure \ref{fig1-Sw}. In this closed system, the chemical potential gradient becomes a dominant driving force. The pressure and chemical potential contours are illustrated in Figure  \ref{fig1-p} and \ref{fig1-chw}. The numerical results {agree well} with the results in \cite{KounS2022}.

\textbf{Example 3.}\label{ex-2}  
In this example, we simulate a two-phase flow in porous media with rock compressibility. The problem is considered in a closed system within the square region $[0, 10 \text{ m}]^2$. We utilize the data as in \cite{KouWChenS}. The initial distributions of porosity and permeability are illustrated in Figure \ref{fig2-initial}. We take the reference porosity $ \phi_r = 0.175$. The viscosities are taken as $\eta_w = 1$ cP and $\eta_n = 0.5$ cP, respectively. 
For the energy parameters, we take
$$
\sigma_w=\frac{\bar{\sigma}_w}{\sqrt{K_0}}, \quad \sigma_n=\frac{\bar{\sigma}_n}{\sqrt{K_0}}, \quad \sigma_{w n}=\frac{\bar{\sigma}_{w n}}{\sqrt{K_0}}, \quad \sigma_{w s}=\sigma_{n s}=0.
$$
In this example, we take $\bar{\sigma}_w =0.58$ Pa, $\bar{\sigma}_n = 0.05$ Pa, $\bar{\sigma}_{wn} = 0.36 $ Pa, and the relative permeability is given as that in Example 2. We initialize the wetting-phase and non-wetting-phase saturation with a uniform distribution, setting  $S_w^0 = 0.3, S_n^0 = 0.7$. In this example, we simulate the problem in a uniform mesh with $70 \times 70$ grid cells and the time step size is taken as $\tau=0.001$ day.

\begin{figure}[htbp]
	\centering
	\subfigure[Initial porosity]{
		\centering
		\includegraphics[width=1.9in]{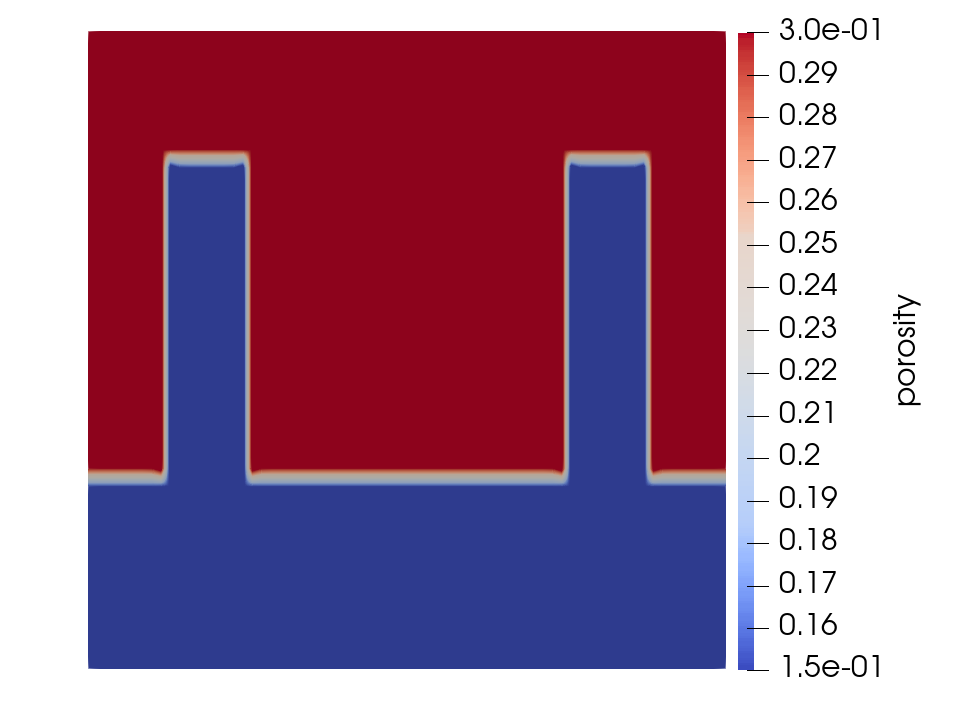}
		\label{fig21}
	}
\subfigure[Initial permeability]{
				\centering
				\includegraphics[width=1.9in]{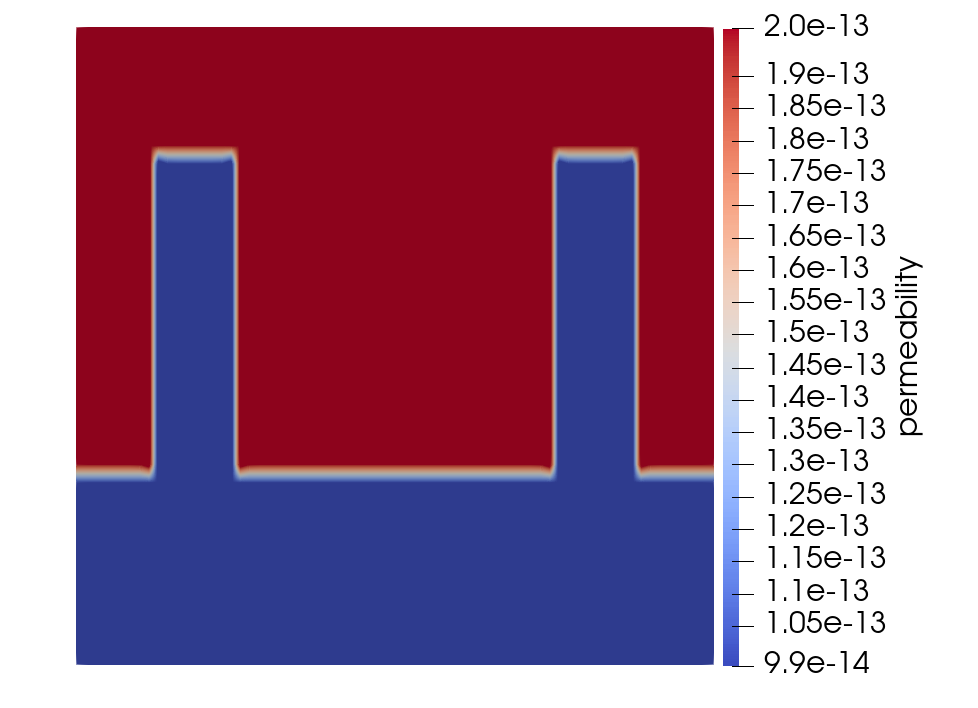}
	}
	\caption{Initial distributions of porosity and permeability in Example 3.} \label{fig2-initial}
\end{figure}

\begin{figure}[htbp]
		\centering
		\includegraphics[width=3.9in]{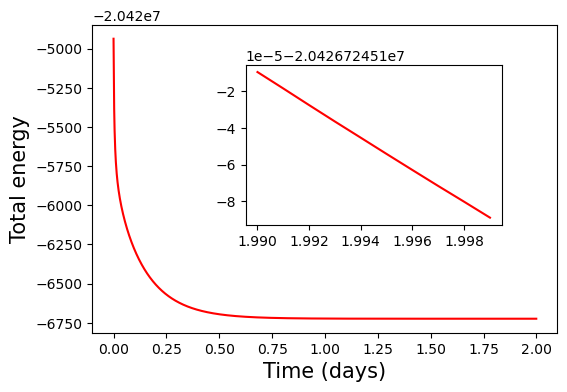}
		\label{fig22}
	\caption{Energy dissipation with time in Example 3.} \label{fig2-energy}
\end{figure}

\begin{figure}[htbp]
	\centering
	\includegraphics[width=5cm]{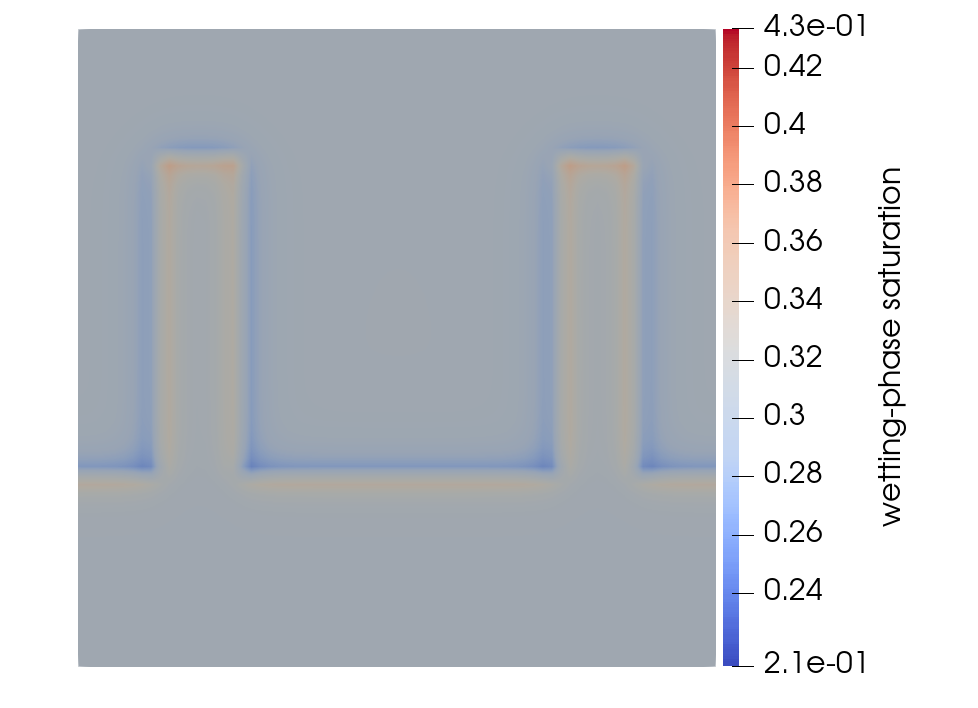}
	\includegraphics[width=5cm]{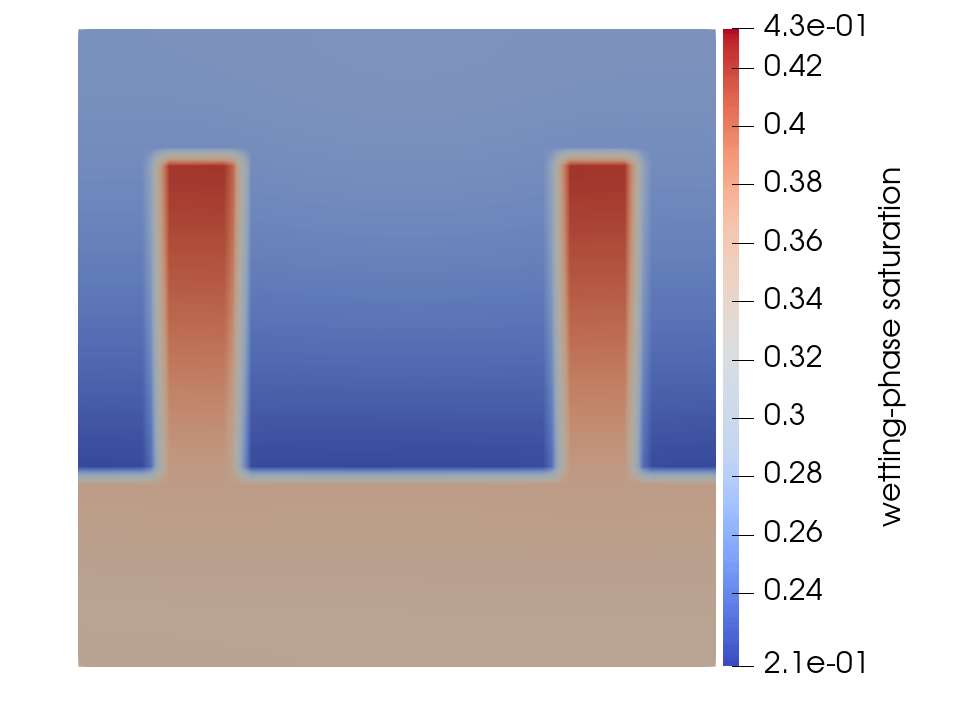}
	
	\includegraphics[width=5cm]{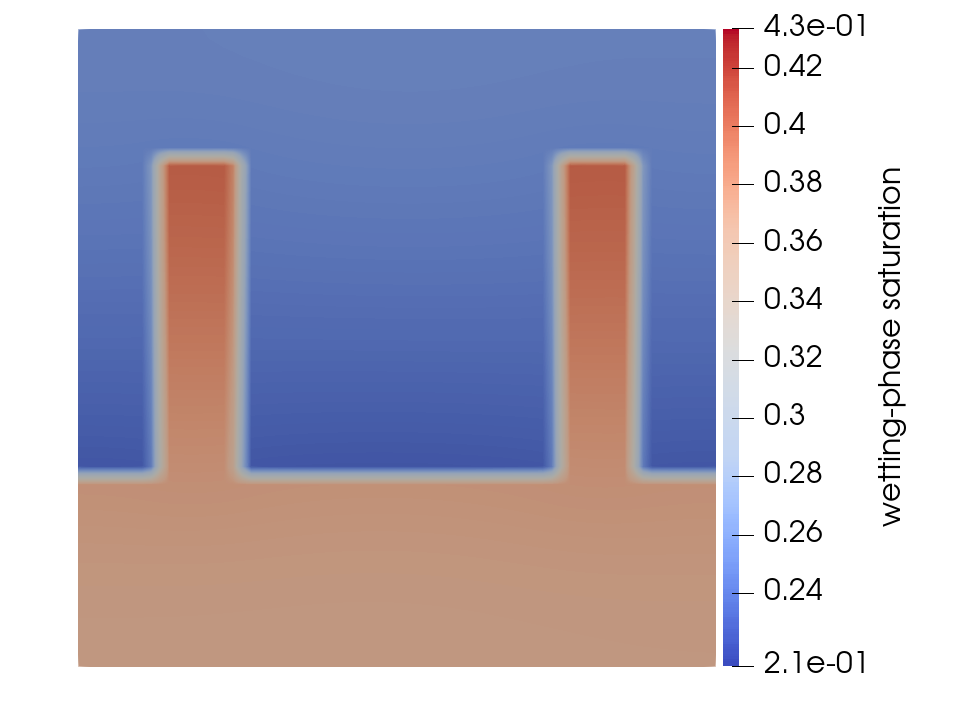}
	\includegraphics[width=5cm]{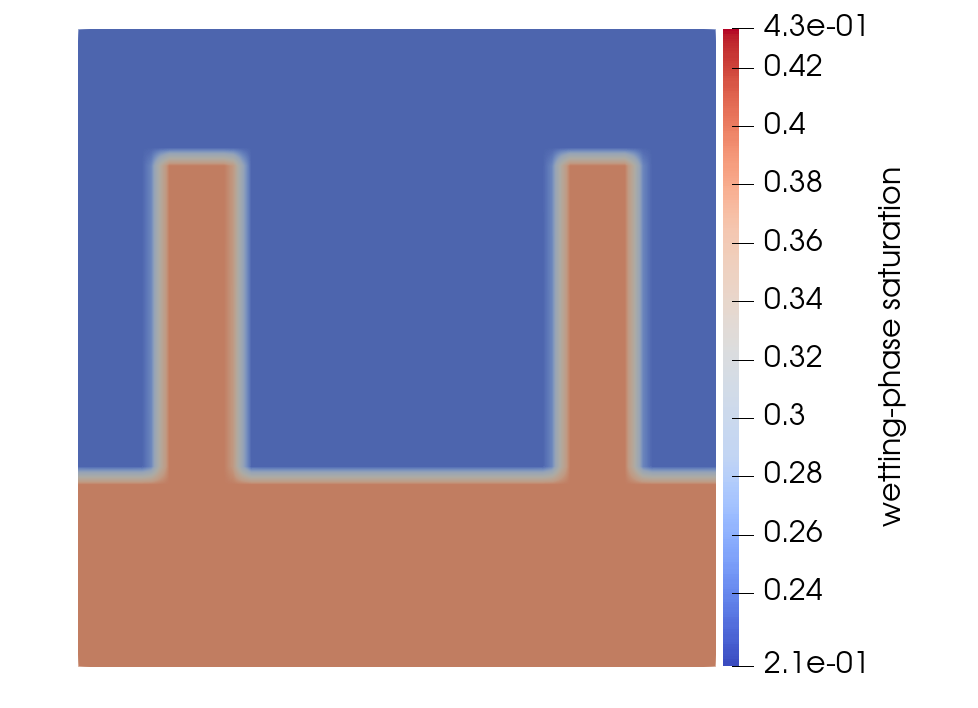}
	\caption{Distributions of wetting-phase saturation at different times in Example 3. Top-left: $t$ = 0.1 day. Top-right: $t$ = 0.2 day. Bottom-left: $t$ = 0.5 day. Bottom-right: $t$ = 2 days.}\label{fig2-Sw}
\end{figure}	

\begin{figure}[htbp]
	\centering
	\includegraphics[width=5cm]{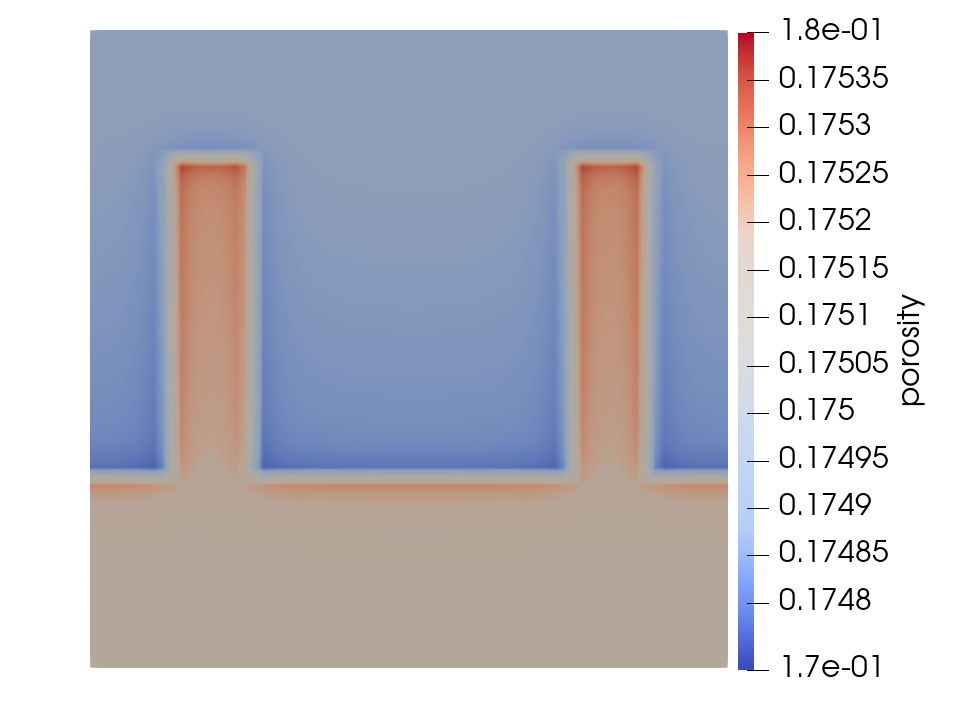}
	\includegraphics[width=5cm]{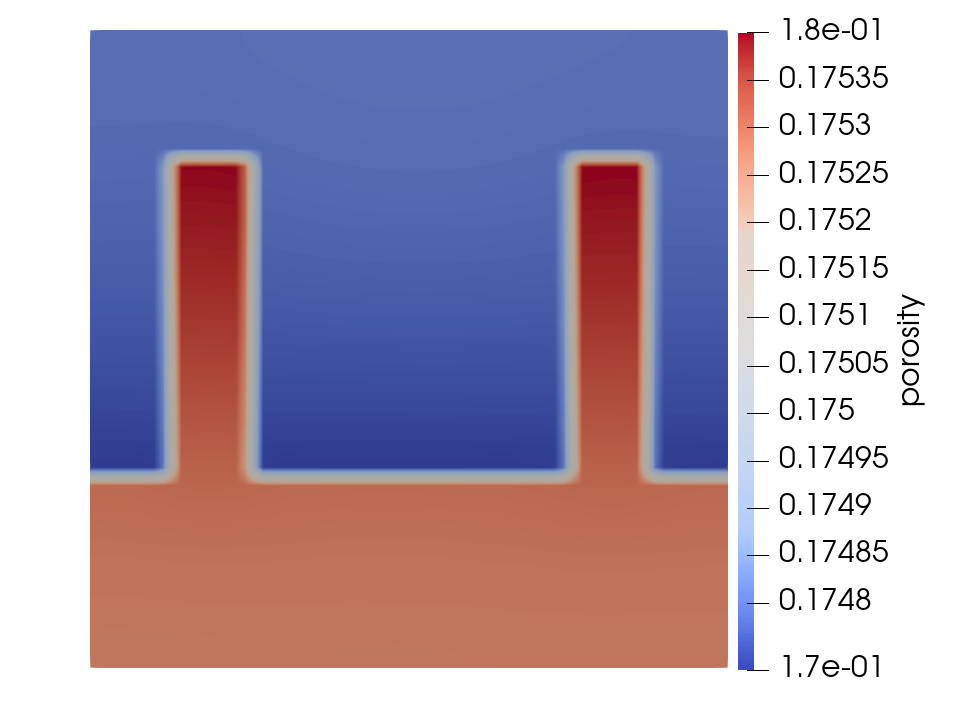}
	
	\includegraphics[width=5cm]{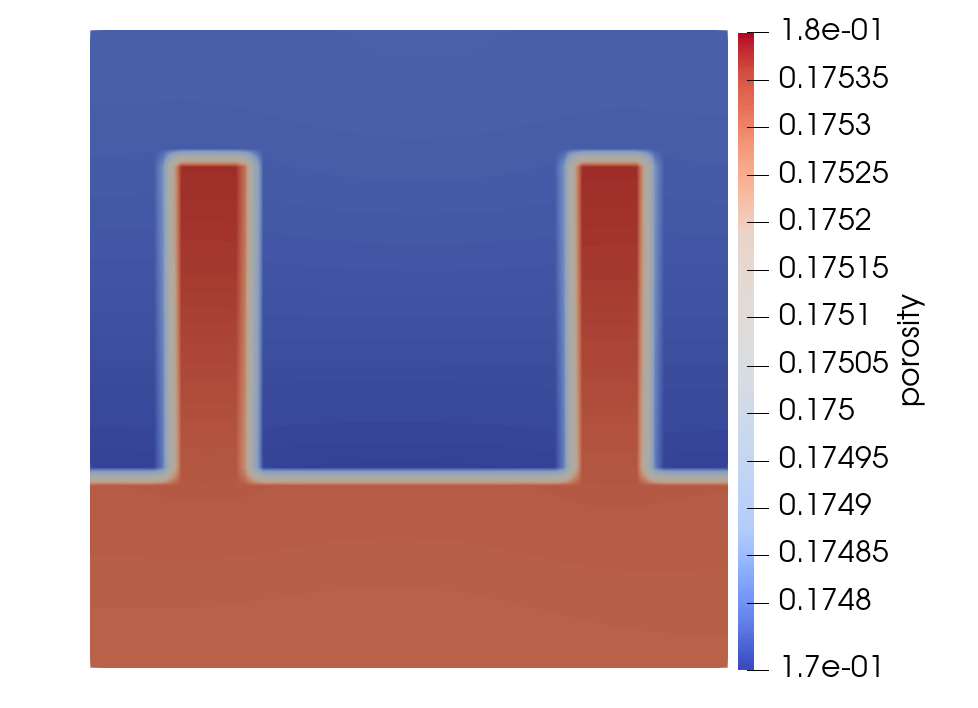}
	\includegraphics[width=5cm]{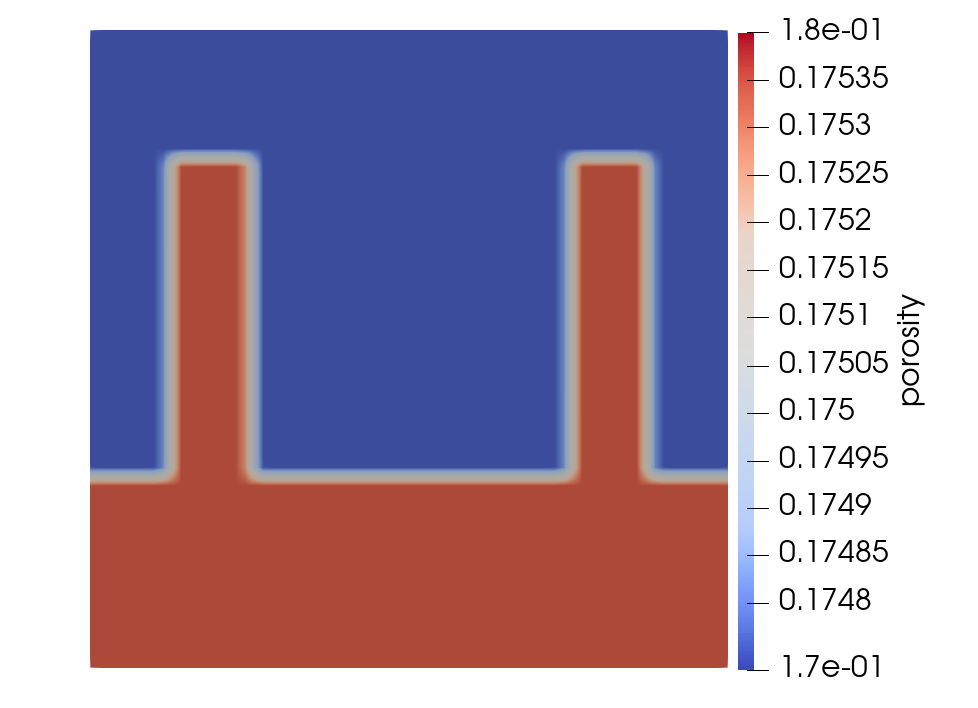}
	\caption{Distributions of porosity at different times in Example 3. Top-left: $t$ = 0.1 day. Top-right: $t$ = 0.2 day. Bottom-left: $t$ = 0.5 day. Bottom-right: $t$ = 2 days.}\label{fig2-phi}
\end{figure}

\begin{figure}[htbp]
	\centering
	\includegraphics[width=5cm]{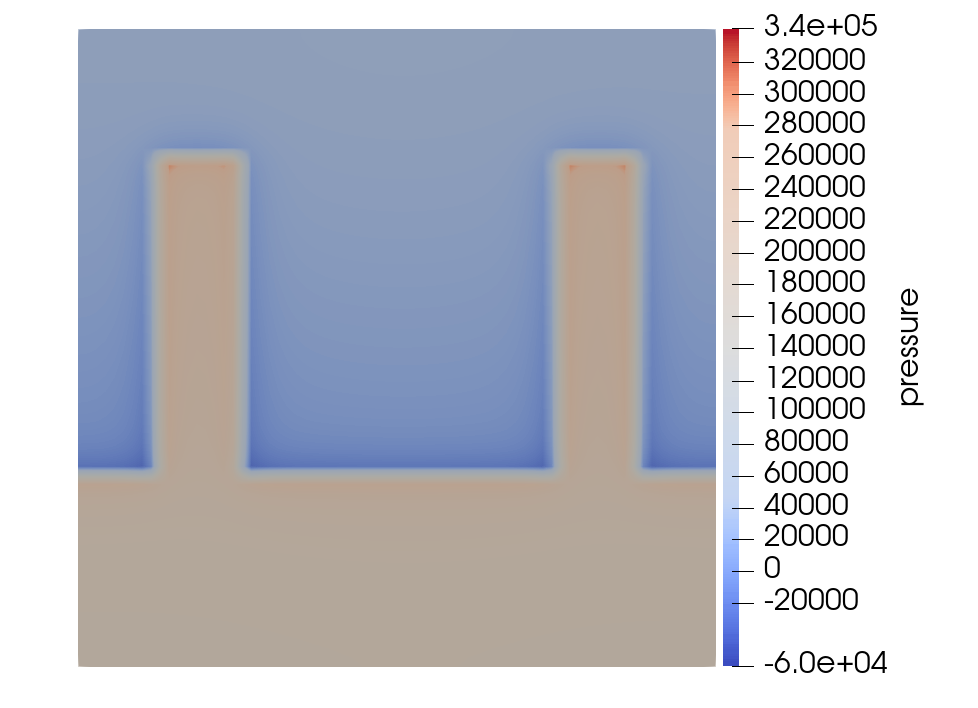}
	\includegraphics[width=5cm]{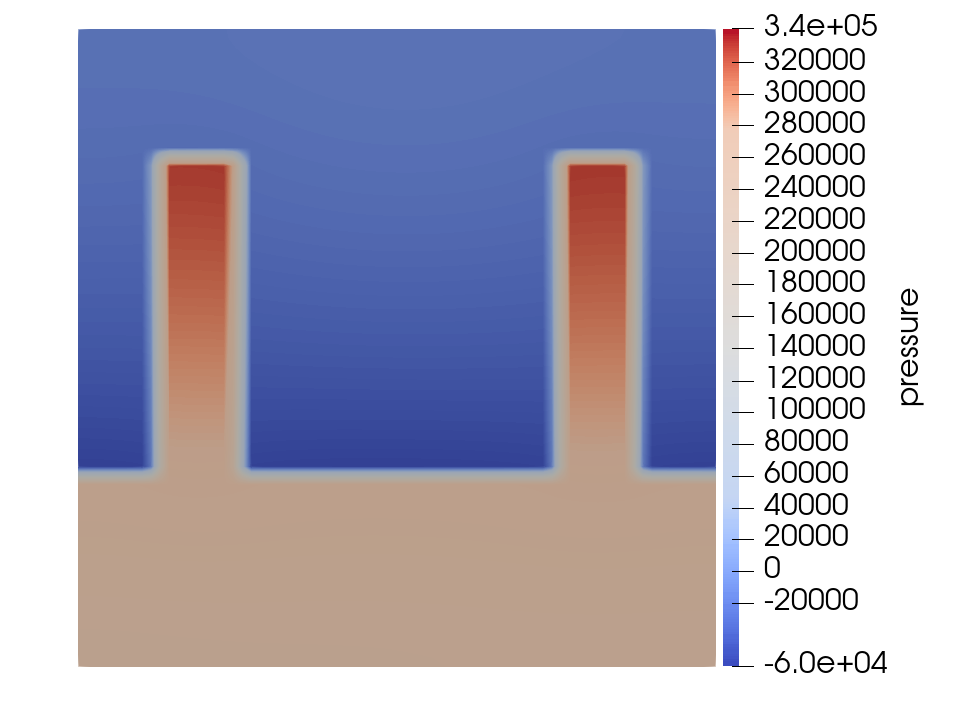}
	
	\includegraphics[width=5cm]{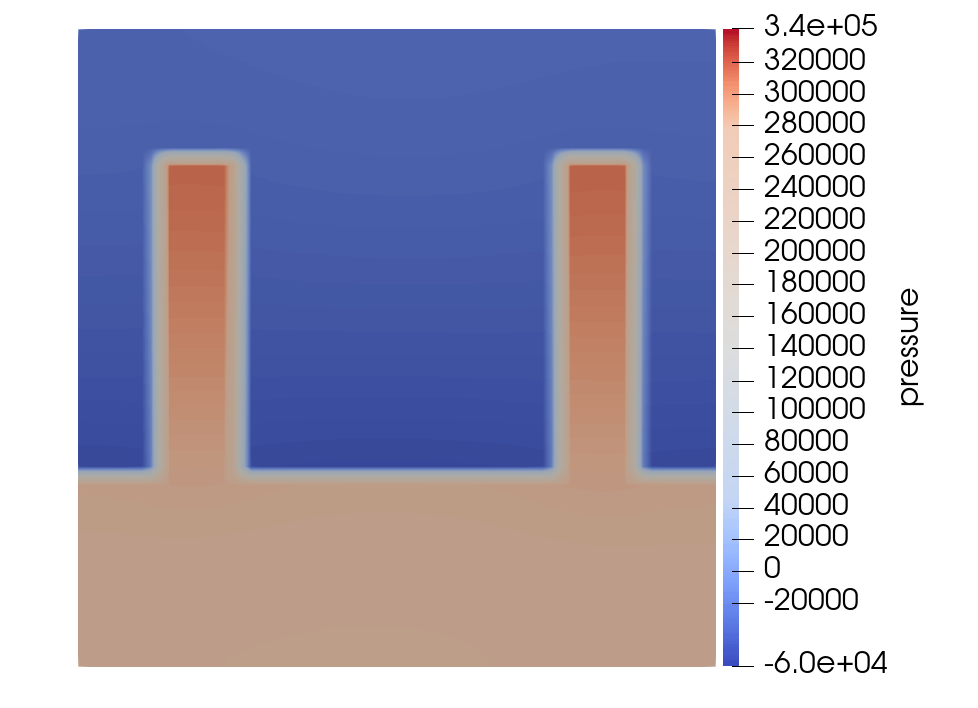}
	\includegraphics[width=5cm]{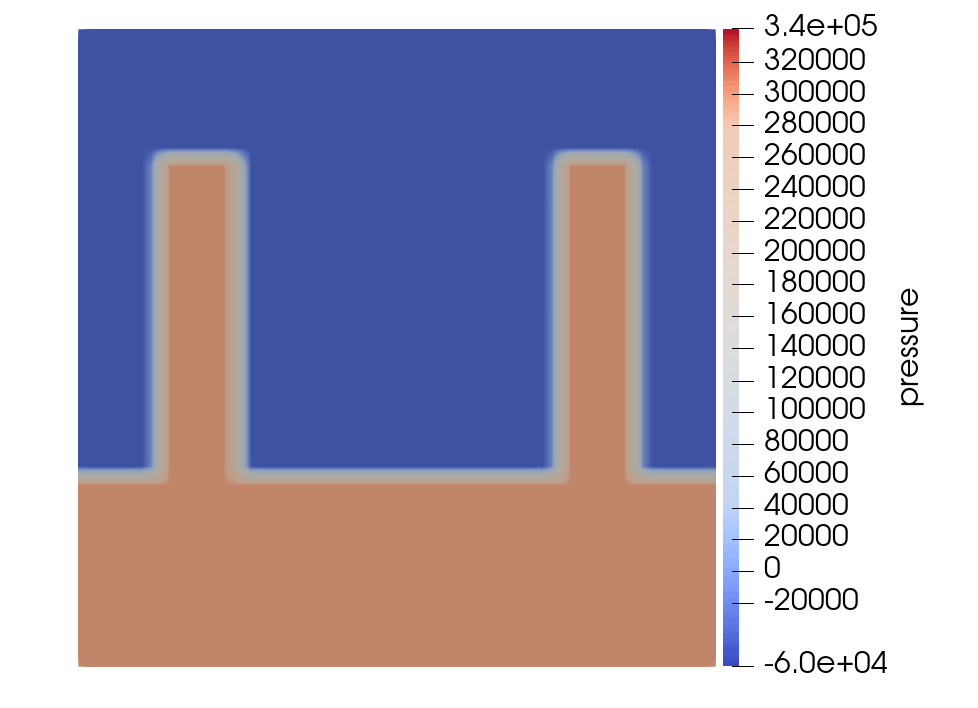}
	\caption{Distributions of effective pressure at different times in Example 3. Top-left: $t$ = 0.1 day. Top-right: $t$ = 0.2 day. Bottom-left: $t$ = 0.5 day. Bottom-right: $t$ = 2 days.}\label{fig2-pressure}
\end{figure} 

Figure \ref{fig2-energy} depicts the total free energy decreases monotonically with time until a steady state is reached. In Figure \ref{fig2-Sw}, we illustrate the wetting-phase saturation fluid flows from high-permeability regions to low-permeability regions until an equilibrium state is reached. As shown in Figure \ref{fig2-phi} and Figure \ref{fig2-pressure}, the effective pore pressure significantly increases in the low-permeability region, leading to an increase in the porosity of the original low-permeability region and a decrease in the porosity of the remaining regions. The numerical results also {agree well} with the results in \cite{KouWChenS}.

\section{Conclusion}
 In conclusion, we introduce a novel framework for designing {structure}-preserving numerical schemes applicable to a wide range of  dissipative systems. Our approach centres on leveraging the Onsager variational principle as an approximation tool. Initially,  we  show that the Onsager principle yields essential dynamic equations for both conservative and non-conservative quantities, including notable  examples such as  phase field equations (e.g., the Allen-Chan or Cahn-Hilliard equations), the Fokker-Planck equation, the PNP equation, and equations governing porous media flows, etc. Subsequently,  we illustrate how this variational principle offers a natural and unified methodology for deriving discrete-time schemes tailored to these equations. These schemes are founded upon   the minimization of the discrete Rayleighian functional. While some schemes align with some existing methods like the JKO scheme in specific scenarios. However, direct application of the JKO scheme is uncommon due to the computational challenges associated with computing Wasserstein distances. Our analysis demonstrates 
that our schemes uphold crucial system properties such as mass conservation and energy dissipation structures. Moreover, our approach allows for 
flexible spacial discretization choices. We provide numerical experiments to validate the effectiveness of our method.

\bibliographystyle{plain}
\bibliography{onsager}
\end{document}